\documentclass[11pt]{article}

\usepackage{amsfonts} 
\usepackage{amsmath,hhline,epsfig,latexsym}
\usepackage{amssymb}

\usepackage{graphicx}
\usepackage[francais,english]{babel}

\usepackage{hyperref}
\usepackage{color}

\newtheorem{thm}{Theorem}
\newtheorem{propo}{Proposition}

\newtheorem{rmq}{Remark}
\newtheorem{lemma}{Lemma}

\newenvironment{proof}[1][Proof]{\noindent\textbf{#1:} }{\hfill$\Box$}

\def\dis{\displaystyle}
\def\Om{\Omega}
\def\om{\omega}

\newcommand{\eps}{\varepsilon}
\newcommand{\Fin}{\hfill$\Box$}

\newcommand{\N}{\mbox{$I \kern -4pt N$}}

\newcommand{\Q}{\mbox{$Q \kern -8pt I$}}
\newcommand{\R}{\mbox{$I \kern -4pt R$}}
\newcommand{\C}{\mbox{$C \kern -8pt I$}}

\newcommand{\mat}[1]{\mbox{\boldmath{$#1$}}}

\newcommand{\yvec}{\mathbf{y}}
\newcommand{\zvec}{\mathbf{z}}
\newcommand{\Fvec}{\mathbf{F}}
\newcommand{\fvec}{\mathbf{f}}

\newcommand{\Wvec}{\mathbf{W}}

\newcommand{\Kvec}{\mathbf{K}}
\newcommand{\vvec}{\mathbf{v}}
\newcommand{\hvec}{\mathbf{h}}
\newcommand{\kvec}{\mathbf{k}}
\newcommand{\ovec}{\mathbf{0}}
\newcommand{\uvec}{\mathbf{u}}
\newcommand{\mvec}{\mathbf{m}}
\newcommand{\xvec}{\mathbf{x}}
\newcommand{\Hvec}{\mathbf{H}}

\newcommand{\Uvec}{\mathbf{U}}
\newcommand{\Vvec}{\mathbf{V}}
\newcommand{\Yvec}{\mathbf{Y}}
\newcommand{\Zvec}{\mathbf{Z}}
\newcommand{\Pvec}{\mathbf{P}}
\newcommand{\Lvec}{\mathbf{L}}
\newcommand{\Avec}{\mathbf{A}}
\newcommand{\nvec}{\mathbf{n}}
\newcommand{\wvec}{\mathbf{w}}

\def\R{{\bf R}}
\begin{document}

\title{\textbf{Uniform local null control of the Leray-$\alpha$ model}}

\author{
	F\'agner D. \textsc{Araruna}\thanks{Departamento de Matem\'{a}tica, Universidade Federal da Para\'iba, 58051-900, Jo\~{a}o Pessoa--PB, Brasil,
	E-mail: {\tt fagner@mat.ufpb.br}. Partially supported by INCTMat, CAPES and CNPq (Brazil).},\ \
	Enrique \textsc{Fern\'andez-Cara}\thanks{Dpto.\ EDAN, University of Sevilla, Aptdo.~1160, 41080~Sevilla, Spain.
	E-mail: {\tt cara@us.es}. Partially supported by CAPES (Brazil) and grants MTM2006-07932, MTM2010-15592 (DGI-MICINN, Spain).},\\
	Diego A. \textsc{Souza}\thanks{Dpto.\ EDAN, University of Sevilla, 41080~Sevilla, Spain 
	E-mail: {\tt desouza@us.es}. Partially supported by grant MTM2010-15592 (DGI-MICINN, Spain).}}

\date{}
% The correct date

\maketitle

\begin{abstract}
	This paper deals with the distributed and boundary controllability of the so called Leray-$\alpha$ model. This is a regularized variant of the
	Navier-Stokes system ($\alpha$ is a small positive parameter) that can also be viewed as a model for turbulent flows. We prove that the
	Leray-$\alpha$ equations	are locally null controllable, with controls bounded independently of $\alpha$. We also prove that, if the initial
	data are sufficiently small, the controls converge as $\alpha \to 0^+$ to a null control of the Navier-Stokes equations. We also discuss some
	other related questions, such as global null controllability, local and global exact controllability to the trajectories, etc.
\end{abstract}


\section{Introduction. The main results}

	Let $\Om\subset \mathbb{R}^N (N=2,3)$ be a bounded connected open set whose boundary $\Gamma$ is of class~$C^2$.
	Let $\om\subset \Om$ be a (small) nonempty open set, let $\gamma\subset \Gamma$ be a (small) nonempty open
	subset of $\Gamma$ and assume that $T>0$. We will use the notation $Q=\Om\times(0,T)$ and
	$\Sigma=\Gamma\times(0,T)$ and we will denote by $\nvec=\nvec(\xvec)$ the outward unit normal to $\Om$ at the points $\xvec\in\Gamma$;
	spaces of $\mathbb{R}^N$-valued functions, as well as their elements, are represented by boldface letters.

	The Navier-Stokes system for a homogeneous viscous incompressible fluid
	(with unit density and unit kinematic viscosity) subject to homogeneous Dirichlet boundary conditions is given by
\begin{equation}\label{NS}
	\left\{
		\begin{array}{lll}
			\yvec_t  - \Delta \yvec +(\yvec\cdot \nabla) \yvec+ \nabla p = \fvec        & \hbox{in} &       Q,     \\
			\nabla \cdot \yvec = 0                                                                             & \hbox{in} &       Q,     \\
			\yvec = \ovec                                                                                          & \hbox{on}&  \Sigma,  \\
			\yvec(0) = \yvec_0                                                                                  & \hbox{in} &  \Om,
		\end{array}
	\right.
\end{equation}
	where $\yvec$ 	(the velocity field) and~$p$ (the pressure) are the unknowns, $\fvec=\fvec(\xvec,t)$ is a forcing
	term and $\yvec_0=\yvec_0(\xvec)$ is a prescribed initial velocity field.
	
	In order to prove the existence of a solution to the Navier-Stokes system, Leray in~\cite{Leray} had the idea of creating a turbulence {\it closure}
	model without enhancing viscous dissipation. Thus, he introduced a ``regularized'' variant of~\eqref{NS} by modifying the nonlinear term as follows:
\[
	\left\{
		\begin{array}{lll}
			\yvec_t  -  \Delta \yvec +(\zvec \cdot \nabla) \yvec+ \nabla p = \fvec      		& \hbox{in} &       Q,     \\
			\nabla \cdot \yvec=0                                  							& \hbox{in} &       Q,
		\end{array}
	\right.
\]
	where $\yvec$ and $\zvec$ are related by
\begin{equation}
	\zvec=\phi_\alpha*\yvec
\end{equation}
	and $\phi_\alpha$ is a smoothing kernel. At least formally, the Navier-Stokes equations are recovered in the limit as
	$\alpha\rightarrow 0^+$, so that $\zvec \rightarrow\yvec$.

	In this paper, we will consider a special smoothing kernel, associated to the Stokes-like operator $\mathbf{Id} + \alpha^2\Avec$,
	where $\Avec$ is the Stokes operator (see Section \ref{Sec2}).
	This leads to the following modification of the Navier-Stokes equations, called  the Leray-$\alpha$ system (see~\cite{HOLM}):
\begin{equation}\label{Lalpha}
	\left\{
		\begin{array}{lll}
			\yvec_t  -  \Delta \yvec +(\zvec \cdot \nabla) \yvec+ \nabla p = \fvec      		& \hbox{in} &       Q,     \\
			\zvec-\alpha^2\Delta \zvec +\nabla \pi=\yvec                                  			& \hbox{in} &       Q,     \\
			\nabla \cdot \yvec=0,~\nabla \cdot \zvec= 0                                    			& \hbox{in} &       Q,     \\
			\yvec = \zvec=\ovec                                                             				& \hbox{on}&  \Sigma,  \\
			\yvec(0) = \yvec_0                                                            					& \hbox{in} &  \Om.
		\end{array}
	\right.
\end{equation}

	In almost all previous works found in the literature, $\Om$ is either the $N$-dimensional torus and the PDE's in~\eqref{Lalpha} are
	completed with periodic boundary conditions or  the whole space $\mathbb{R}^N$. Then, $\zvec$ satisfies an equation of the kind
\begin{equation}\label{Lalpha1}
	\zvec-\alpha^2\Delta \zvec =\yvec
\end{equation}
	and the model is (apparently) slightly different from~\eqref{Lalpha}.
	However, since $\nabla \cdot \yvec= 0$, it is easy to see that \eqref{Lalpha1}, in these cases,  is equivalent to the equation satisfied by $\zvec$
	and~$\pi$ in~\eqref{Lalpha}.
	
	It has been shown in~\cite{HOLM} that, at least for periodic boundary conditions, the numerical solution of the equations in~\eqref{Lalpha}
	matches successfully with empirical data from turbulent channel and pipe flows for a wide range of Reynolds numbers. Accordingly,
	the Leray-$\alpha$ system has become preferable to other turbulence models, since the associated computational cost is
	lower and no introduction of {\it ad~hoc} parameters is required.
	
	In \cite{GIBBON}, the authors have compared the numerical solutions of three different $\alpha$-models useful in turbulence
	modelling (in terms of the Reynolds number associated to a Navier-Stokes velocity field). The results improve as one passes from the
	Navier-Stokes equations to these models and clearly show that the Leray-$\alpha$ system has the best performance. Therefore, it seems
	quite natural to carry out a theoretical analysis of the solutions to~\eqref{Lalpha}.

	We will be concerned with the following controlled systems
\begin{equation}
	\label{CLalpha}
		\left\{
			\begin{array}{lll}
				\yvec_t  -  \Delta \yvec +(\zvec \cdot \nabla) \yvec+ \nabla p = \vvec1_\om  		& \hbox{in} &       Q,     \\
				\zvec-\alpha^2\Delta \zvec +\nabla \pi=\yvec                                               		& \hbox{in} &       Q,    \\
				\nabla \cdot \yvec=0,~\nabla \cdot \zvec= 0                                                 		& \hbox{in} &       Q,     \\
				\yvec = \zvec=\ovec                                                                                     		& \hbox{on}&  \Sigma, \\
				\yvec(0) = \yvec_0                                                                                       		& \hbox{in} &  \Om,
			\end{array}
		\right.
\end{equation}
	and
	\begin{equation}
	\label{BBCLalpha}
		\left\{
			\begin{array}{lll}
				\yvec_t  -  \Delta \yvec +(\zvec \cdot \nabla) \yvec+ \nabla p = \ovec	  		& \hbox{in} &       Q,     \\
				\zvec-\alpha^2\Delta \zvec +\nabla \pi=\yvec                                               		& \hbox{in} &       Q,    \\
				\nabla \cdot \yvec=0,~\nabla \cdot \zvec= 0                                                 		& \hbox{in} &       Q,     \\
				\yvec = \zvec=\hvec1_{\gamma}                                                                 		& \hbox{on}&  \Sigma, \\
				\yvec(0) = \yvec_0                                                                                       		& \hbox{in} &  \Om,
			\end{array}
		\right.
\end{equation}
	where $\vvec = \vvec(\xvec, t)$ (respectively $\hvec=\hvec(\xvec,t)$) stands for the control, assumed to act only in the (small) set $\omega$
	(respectively on $\gamma$) during the whole time interval $(0,T)$. The symbol $1_\om$ (respectively $1_{\gamma}$) stands for the
	characteristic function of~$\om$ (respectively of $\gamma$).
	
	In the applications, the {\it internal control} $\vvec1_\om$ can be viewed as a gravitational or electromagnetic field.
	The {\it boundary control} $\hvec1_\gamma$ is the trace of the velocity field on $\Sigma$.

	\begin{rmq}
		\rm It is completely natural to suppose that $\yvec$ and $\zvec$ satisfy the same boundary conditions on $\Sigma$ since,
		in the limit, we should have $\zvec=\yvec$. Consequently, we will assume that the boundary control $\hvec1_{\gamma}$ acts
		simultaneously on both variables $\yvec$ and $\zvec$.
	\end{rmq}
	
	In what follows, $(\cdot\,,\cdot)$ and $\|\cdot\|$ denote the usual $L^2$ scalar products and norms
	(in~$L^2(\Om)$, $\Lvec^2(\Om)$, $L^2(Q)$, etc.) and $K$, $C$,
	$C_1$, $C_2$, \dots\ denote various positive constants (usually depending on $\om$, $\Om$ and~$T$). Let us recall the
	definitions of some usual spaces in the context of incompressible fluids:
\[
	\begin{array}{c}
		\Hvec = \{\, \uvec\in\Lvec^2(\Om) : \nabla \cdot\uvec = 0~\hbox{in}~ \Om, \ \uvec\cdot\nvec = 0~\hbox{on}~\Gamma \,\}, \\
		\noalign{\smallskip}
		\Vvec = \{\, \uvec\in\Hvec^1_0(\Om) : \nabla\cdot\uvec=0~ \hbox{in}~\Om \,\}.
	\end{array}
\]
	
	Note that, for every $\yvec_0 \in \Hvec$ and every $\vvec \in \Lvec^2(\om \times (0,T))$, there exists a unique solution $(\yvec,p,\zvec,\pi)$ for~\eqref{CLalpha} that satisfies (among other things)
\[
	\yvec, \zvec \in C^0([0,T];\Hvec);
\]
	see~Proposition~\ref{NC-LERAY} below.
	This is in contrast with the lack of uniqueness of the Navier-Stokes system when $N = 3$.
	
	The main goals of this paper are to analyze the controllability properties of~\eqref{CLalpha} and~\eqref{BBCLalpha} and determine the
	way they depend on~$\alpha$ as $\alpha \to 0^+$.

	The null controllability problem for \eqref{CLalpha} at time $T>0$ is the following:
\begin{quote}\it{
	For any $\yvec_0\in \Hvec$, find $\vvec\in \Lvec^2(\om\times(0,T))$ such that the corresponding state~$($the corresponding
	solution to~\eqref{CLalpha}$)$ satisfies
	}
\end{quote}
\begin{equation}\label{null_condition}
	\yvec(T) = \ovec \quad\hbox{\it in} \quad \Om.
\end{equation}

	The null controllability problem for \eqref{BBCLalpha} at time $T>0$ is the following:
\begin{quote}\it{
	For any $\yvec_0\in \Hvec$, find $\hvec\in \Lvec^2(0,T; \Hvec^{-1/2}(\gamma))$
	with $\int_{\gamma} \hvec \cdot \nvec \,d\Gamma =0$ and an associated state $($the corresponding solution to~\eqref{BBCLalpha}$)$ satisfying
	$$
\yvec, \zvec \in C^0([0,T];\Lvec^2(\Om))
	$$
	and~\eqref{null_condition}.
	}
\end{quote}
	
	Recall that, in the context of the Navier-Stokes equations, J.-L.~Lions conjectured in~\cite{Lions2} the global distributed and boundary
	approximate controllability; since then, the controllability of these equations has been intensively studied, but for the moment only partial
	results are known.
	
	Thus, the global approximate controllability of the two-dimensional Navier-Stokes equations with Navier slip boundary conditions was
	obtained by~Coron in~\cite{CORON}.
	Also, by combining results concerning global and local controllability, the global null controllability for the Navier-Stokes system on a
	two-dimensional manifold without boundary was established in~Coron and~Fursikov~\cite{CoronF}; see also Guerrero {\it et al.} \cite{G-APROX}
	for another global controllability result.
	
	The local exact controllability to bounded trajectories has been obtained by~Fursikov and~Imanuvilov
	~\cite{FURS-IMANU,F-IM2}, Imanuvilov~\cite{IMANU} and~Fern\'an\-dez-Cara~{\it et al.}~\cite{FC-G-P} under various circumstances;
	see~Guerrero~\cite{G} and~Gonz\'alez-Burgos~{\it et al.}~\cite{GBGP} for similar results related to the Boussinesq system.
	Let us also mention~\cite{C-G,Co-Gu,Coron-Lissy,FC-G-P-2}, where analogous results are obtained with a reduced number of scalar controls.
	
	For the (simplified) one-dimensional viscous Burgers model, positive and negative results can be found in~\cite{FC-G,G-G,G-I};
	see also~\cite{E-G-G-P}, where the authors consider the one-dimensional compressible Navier-Stokes system.

	Our first main result in this paper is the following:
	
\begin{thm}\label{NC-LERAY}
	There exists $\epsilon>0~(\hbox{independent of }~\alpha)$ such that, for each $\yvec_0 \in \Hvec$ with $\|\yvec_0\| \leq \epsilon$, there exist
	controls ~$\vvec_\alpha\in L^\infty(0,T;\Lvec^2(\omega))$ such that the associated solutions to \eqref{CLalpha} fulfill ~\eqref{null_condition}. 				 Furthermore, these controls can be found satisfying the estimate
\begin{equation}\label{v-unif}
	\|\vvec_\alpha\|_{L^\infty(\Lvec^2)}\leq C,
\end{equation}
	where $C$ is also independent of $\alpha$.
\end{thm}

	Our second main result is the analog of Theorem~\ref{NC-LERAY} in the framework of boundary controllability. It is the following:		

\begin{thm}\label{NC-LERAY-BOUNDARY}
	There exists $\delta>0~(\hbox{independent of }~\alpha)$ such that, for each $\yvec_0 \in \Hvec$ with $\|\yvec_0\| \leq \delta$,
	there exist controls $\hvec_\alpha\in L^\infty(0,T;\Hvec^{-1/2}(\gamma))$ with $\int_{\gamma} \hvec_\alpha\cdot \nvec \,d\Gamma =0$
	and associated solutions to~\eqref{BBCLalpha} that fulfill~\eqref{null_condition}. Furthermore, these controls can be found satisfying the estimate
\begin{equation}\label{v-unif-boundary}
	\|\hvec_\alpha\|_{L^\infty(H^{-1/2})}\leq C,
\end{equation}
	where $C$ is also independent of $\alpha$.
	
\end{thm}

	The proofs rely on suitable fixed-point arguments. The underlying idea has applied to many other nonlinear control problems.
	However, in the present cases, we find two specific difficulties:
	
\begin{itemize}

\item In order to find spaces and fixed-point mappings appropriate for Schauder's Theorem, the initial state $\yvec_0$ must be regular enough.
	Consequently, we have to establish {\it regularizing properties} for~\eqref{CLalpha} and~\eqref{BBCLalpha};
	see~Lemmas~\ref{regularity} and~\ref{regularity_BC} below.
	
\item For the proof of the uniform estimates~\eqref{v-unif} and~\eqref{v-unif-boundary}, careful estimates of the null controls and associated states of some particular linear problems are needed.

\end{itemize}

	We will also prove results concerning the controllability in the limit, as~$\alpha\to 0^+$.
	It will be shown that the null-controls for \eqref{CLalpha} can be chosen in such a way that they converge to null-controls for the
	Navier-Stokes system
\begin{equation}
	\label{C-NS}
		\left\{
			\begin{array}{lll}
				\yvec_t  -  \Delta \yvec +(\yvec \cdot \nabla) \yvec+ \nabla p = \vvec1_\om	& \hbox{in}&       Q,     \\
				\nabla \cdot \yvec=0                                                 						& \hbox{in}&       Q,     \\
				\yvec = \ovec                                                                       	              		& \hbox{on}&  \Sigma, \\
				\yvec(0) = \yvec_0                        				                         		& \hbox{in}&  \Om.
			\end{array}
		\right.
\end{equation}

	Also, it will be seen that the null-controls for \eqref{BBCLalpha} can be chosen such that they converge to boundary
	null-controls for the Navier-Stokes system
\begin{equation}
	\label{BBC-NS}
		\left\{
			\begin{array}{lll}
				\yvec_t  -  \Delta \yvec +(\yvec \cdot \nabla) \yvec+ \nabla p = \ovec		& \hbox{in}&       Q,     \\
				\nabla \cdot \yvec=0                                                 					& \hbox{in}&       Q,     \\
				\yvec = \hvec1_\gamma                                                  	              		& \hbox{on}&  \Sigma, \\
				\yvec(0) = \yvec_0                        				                         		& \hbox{in}&  \Om.
			\end{array}
		\right.
\end{equation}

	More precisely, our third and fourth main results are the following:
	
\begin{thm}\label{CONVERGENCE}
	Let $\epsilon>0$ be furnished by Theorem~$\ref{NC-LERAY}$. Assume that $\yvec_0\in \Hvec$
	and $\|\yvec_0\| \leq \epsilon$, let $\vvec_\alpha$ be a null control for~\eqref{CLalpha} satisfying \eqref{v-unif} and let
	$(\yvec_\alpha,p_\alpha,\zvec_\alpha,\pi_{\alpha})$ be the associated state. Then, at least for a subsequence, one has
\[
	\begin{array}{c}
		\vvec_\alpha\to \vvec\hbox{ weakly-$*$} \hbox{ in } L^\infty(0,T;\Lvec^2(\om)), \\
		\zvec_\alpha \to \yvec \hbox{ and } \yvec_\alpha\to \yvec \hbox{ strongly in } \Lvec^2(Q),
	\end{array}
\]
	as $\alpha\to 0^+$, where $(\yvec,\vvec)$ is, together with some $p$, a state-control pair for~\eqref{C-NS} satisfying~\eqref{null_condition}.
\end{thm}

\begin{thm}\label{CONVERGENCE-BOUNDARY}
	Let $\delta>0$ be furnished by Theorem~$\ref{NC-LERAY-BOUNDARY}$. Assume that $\yvec_0\in \Hvec$
	and $\|\yvec_0\| \leq \delta$, let $\hvec_\alpha$ be a null control for~\eqref{BBCLalpha} satisfying \eqref{v-unif-boundary} and let
	$(\yvec_\alpha,p_\alpha,\zvec_\alpha,\pi_{\alpha})$ be the associated state. Then, at least for a subsequence, one has
\[
	\begin{array}{c}
		\hvec_\alpha\to \hvec\hbox{ weakly-$*$} \hbox{ in } L^\infty(0,T;H^{-1/2}(\gamma)), \\
		\zvec_\alpha \to \yvec \hbox{ and } \yvec_\alpha\to \yvec \hbox{ strongly in } \Lvec^2(Q),
	\end{array}
\]
	as $\alpha\to 0^+$, where $(\yvec,\hvec)$ is, together with some $p$, a state-control pair for~\eqref{BBC-NS} satisfying~\eqref{null_condition}.
\end{thm}

	The rest of this paper is organized as follows.
	In Section~\ref{Sec2}, we will recall some properties of the Stokes operator and we will prove some results concerning the existence, uniqueness and
	regularity of the solution to~\eqref{Lalpha}.
	Section~\ref{Sec3} deals with the proofs of Theorems~\ref{NC-LERAY} and~\ref{CONVERGENCE}.
	Section~\ref{Sec4} deals with the proofs of Theorems~\ref{NC-LERAY-BOUNDARY} and~\ref{CONVERGENCE-BOUNDARY}.
	Finally, in~Section~\ref{Sec5}, we present some additional comments and open questions.

%%%%%%%%%%%%%%%%%%%%%%%%%%%%%%%%%%%%%%%%
%%%%%%%%%%%%%%%%%%%%%%%%%%%%%%%%%%%%%%%%
%%%%%%%%       SECTION 2        %%%%%%%%%%%%%%%%%%%%%
%%%%%%%%%%%%%%%%%%%%%%%%%%%%%%%%%%%%%%%%
%%%%%%%%%%%%%%%%%%%%%%%%%%%%%%%%%%%%%%%%

\section{Preliminaries}\label{Sec2}

	 In this section, we will recall some properties of the Stokes operator. Then, we will prove that the Leray-$\alpha$ system is well-posed. Also, we
	 will recall the Carleman inequalities and null controllability properties of the Oseen system.

%%%%%%%%%%%%%%%%%%%%%%%%%%%%%%%%%%%%%%%%
%%%%%%%%%%%%%%%%%%%%%%%%%%%%%%%%%%%%%%%%
%%%%%%%%       SECTION 2.1             %%%%%%%%%%%%%%%%%%
%%%%%%%%%%%%%%%%%%%%%%%%%%%%%%%%%%%%%%%%
%%%%%%%%%%%%%%%%%%%%%%%%%%%%%%%%%%%%%%%%

\subsection{The Stokes operator}

         Let $\Pvec: \Lvec^2(\Omega) \mapsto \Hvec$ be the orthogonal projector, usually known as the {\it Leray Projector}.
	Recall that $\Pvec$ maps $\Hvec^s(\Om)$ into $\Hvec^s(\Om)\cap \Hvec$ for all $s\geq0$.
	
	We will denote by $\Avec$ the {\it Stokes operator}, i.e.~the self-adjoint operator in $\Hvec$ formally given by $\Avec=-\Pvec\Delta$.
	For any $\uvec\in D(\Avec):=\Vvec\cap\Hvec^2(\Om)$ and any $\wvec\in \Hvec$, the identity $\Avec\uvec=\wvec$ holds if and only if
\[
	(\nabla\uvec,\nabla\vvec)=(\wvec,\vvec) \quad \forall\, \vvec\in\Vvec.
\]

	It is well known that $\Avec : D(\Avec) \mapsto \Hvec$ can be inverted and~its inverse~$\Avec^{-1}$ is self-adjoint, compact and~positive.
	Consequently, there exists a nondecreasing sequence of positive numbers $\lambda_j$ and an associated orthonormal basis of $\Hvec$,
	denoted by $(\wvec_j)_{j=1}^{\infty}$, such that
\[
	\Avec \wvec_j = \lambda_j\wvec_j \quad \forall j \geq 1.
\]

	Accordingly we can introduce the real powers of the Stokes operator.
	Thus, for any $r \in \mathbb{R}$, we set
\[
	D(\Avec^r)= \{\, \uvec\in \Hvec:  \uvec =\sum_{j=1}^{\infty} u_j\wvec_j, \hbox{ with}~\sum_{j=1}^{\infty}\lambda_j^{2r} |u_j|^2< +\infty \,\}
\]
	and
\[
	\Avec^r \uvec =\sum_{j=1}^{\infty}\lambda_j^r u_j\wvec_j, \quad \forall\, \uvec =\sum_{j=1}^{\infty} u_j\wvec_j \in D(\Avec^r).
\]

	Let us present a result concerning the domains of the powers of the Stokes operator.	
\begin{thm}\label{embb}
	Let $r\in \mathbb{R}$ be given, with $-{1\over2}<r< 1$. Then
\[
	\begin{array}{l}
	\noalign{\smallskip}\dis
	D(\Avec^{r/2})= \Hvec^{r}(\Om)\cap \Hvec~\hbox{ whenever }      ~  -\frac{1}{2}<r< \frac{1}{2}, \\
	\noalign{\smallskip}\dis
	D(\Avec^{r/2})= \Hvec^{r}_0(\Om)\cap \Hvec~\hbox{ whenever }  ~  \frac{1}{2}\leq r\leq 1.
	\end{array}
\]
	Moreover, $\uvec\mapsto (\uvec,\Avec^r\uvec)^{1/2}$ is a Hilbertian norm in $D(\Avec^{r/2})$, equivalent to the usual Sobolev $\Hvec^r$-norm.
	In other words, there exist constants $c_1(r),c_2(r)>0$ such that
\[
	\dis c_1(r) \|\uvec\|_{\Hvec^r}\leq(\uvec,\Avec^r\uvec)^{1/2}\leq c_2(r) \|\uvec\|_{\Hvec^r}~~\forall \uvec\in D(\Avec^{r/2}).
\]
\end{thm}

	The proof of Theorem \ref{embb} can be found in \cite{FU-MORI}.
	Notice that, in view of the interpolation $K$-method of Lions and Peetre, we have $D(\Avec^{r/2})=D((-\Delta)^{r/2})\cap \Hvec$.
	Hence, thanks to an explicit description of $D((-\Delta)^{r/2})$, the stated result holds.

	Now, we are going to recall an important property of the semigroup of contractions $e^{-t\Avec}$ generated by $\Avec$, see~\cite{FUJITA-KATO-1}:
	
\begin{thm} \label{estimatesemigroup}
	For any $r > 0$, there exists $C(r) > 0$ such that
\begin{equation}\label{t-r}
	\|\Avec^r e^{-t\Avec}\|_{\mathcal{L}(\Hvec;\Hvec)}\leq C(r) \, t^{-r}~~\forall \, t>0.
\end{equation}
\end{thm}

	In order to prove \eqref{t-r}, it suffices to observe that, for any $ \uvec =\sum_{j=1}^{+\infty} u_j\wvec_j\in \Hvec$, one has
\[
		\Avec^re^{-t\Avec}\uvec=\sum_{j=1}^{+\infty} \lambda_j^re^{-t\lambda_j} u_j\wvec_j.
\]
	Consequently,
\[
		\|\Avec^re^{-t\Avec}\uvec\|^2=\sum_{j=1}^{+\infty}\left|\lambda_j^r e^{-t\lambda_j}u_j\right|^2
		\leq\left(\max\limits_{\lambda\in\mathbb{R}}\lambda^r e^{-t\lambda}\right)^2\|\uvec\|^2
\]
	and, since $\max\limits_{\lambda\in\mathbb{R}}\lambda^r e^{-t\lambda}=(r/e)^r \, t^{-r}$, we get easily \eqref{t-r}.

%%%%%%%%%%%%%%%%%%%%%%%%%%%%%%%%%%%%%%%%
%%%%%%%%%%%%%%%%%%%%%%%%%%%%%%%%%%%%%%%%
%%%%%%%% SECTION 2.2       %%%%%%%%%%%%%%%%%%%%%%
%%%%%%%%%%%%%%%%%%%%%%%%%%%%%%%%%%%%%%%%
%%%%%%%%%%%%%%%%%%%%%%%%%%%%%%%%%%%%%%%%

\subsection{Well-posedness for the Leray-$\alpha$ system}

	Let us see that, for any $\alpha > 0$, under some reasonable conditions on~$\fvec$ and~$\yvec_0$, the Leray-$\alpha$ system \eqref{Lalpha}
	possesses a unique global weak solution.
	Before this, let us introduce $\sigma_N$ given by
\[
	\sigma_N =
	\left\{
		\begin{array}{ll}
			2 \ & \ \text{if $N=2$,} \\
			4/3 \ & \ \text{if $N=3$. }
		\end{array}
	\right.
\]
	Then, we have the following result:
	
\begin{propo}\label{G-W-E-U-L-alpha}
	Assume that $\alpha>0$ is fixed. Then, for any $\fvec\in L^2(0,T;\Hvec^{-1}(\Om))$ and any $\yvec_0\in \Hvec$,
	there exists exactly one solution $(\yvec_\alpha,p_\alpha,\zvec_\alpha,\pi_{\alpha})$ to~\eqref{Lalpha}, with
\begin{equation}\label{space-weak}
	\begin{array}{c}
		\noalign{\smallskip}\dis
		\yvec_\alpha\in L^2(0, T;\Vvec) \cap C^0([0, T];\Hvec),~ (\yvec_\alpha)_t\in L^2(0,T;\Vvec'), \\
		\noalign{\smallskip}\dis
		\zvec_\alpha\in L^2(0,T;D(\Avec^{3/2})) \cap C^0([0, T];D(\Avec)).
	\end{array}
\end{equation}
	Furthermore, the following estimates hold:
\begin{equation}\label{yalpha-ineqq}
	\begin{alignedat}{2}
		\noalign{\smallskip}\dis
		\|\yvec_\alpha\|_{L^2(\Vvec)}+\|\yvec_\alpha\|_{C^0([0,T];\Hvec)}\leq&~CB_0(\yvec_0,\fvec), 							\\
		\noalign{\smallskip}\dis
		\|(\yvec_\alpha)_t\|_{L^{\sigma_N}(\Vvec')}\leq&~CB_0(\yvec_0,\fvec)(1+B_0(\yvec_0,\fvec)), 							\\
		\noalign{\smallskip}\dis
		\|\zvec_\alpha\|^2_{L^\infty(\Hvec)}+2\alpha^2\|\zvec_\alpha\|^2_{L^\infty(\Vvec)}\leq&~CB_0(\yvec_0,\fvec)^2,			 \\
		\noalign{\smallskip}\dis
		2\alpha^2\|\zvec_\alpha\|^2_{L^\infty(\Vvec)}+\alpha^4\|\zvec_\alpha\|^2_{L^\infty(D(\Avec))}\leq&~CB_0(\yvec_0,\fvec)^2.
	\end{alignedat}
\end{equation}
	Here, $C$ is independent of $\alpha$ and we have introduced the notation
\[
	B_0(\yvec_0,\fvec):=\|\yvec_0\| + \|\fvec\|_{L^2(\Hvec^{-1})}.
\]
\end{propo}

\begin{proof}
	The proof follows classical and rather well known arguments;
	see for instance~\cite{DAUTRAY-JLL,TEMAM}.
	For completeness, they will be recalled.
	
	\textsc{$\bullet$ Existence:} We will reduce the proof to the search of a fixed point of an appropriate mapping $\Lambda_\alpha$.
	\footnote{
	Alternatively, we can prove the existence of a solution by introducing adequate Galerkin approximations and applying
	(classical) compactness arguments.}

	Thus, for each $\overline \yvec\in L^2(0,T;\Hvec)$, let $(\zvec,\pi)$ be the unique solution to
\[
	\left\{
		\begin{array}{lll}
			\zvec-\alpha^2\Delta \zvec +\nabla \pi=\overline\yvec                		& \hbox{in}&       Q,     \\
			\nabla\cdot\zvec=0                                       						& \hbox{in}&       Q,     \\
			\zvec  = \ovec         											&\hbox{on}& \Sigma.
		\end{array}
	\right.
\]
	It is clear that $\zvec \in L^{2}(0,T;D(\Avec))$ and then, thanks to the Sobolev embedding, we have $\zvec \in L^2(0,T;\Lvec^\infty(\Om))$.
	Moreover, the following estimates are satisfied:
\[
	\begin{alignedat}{2}
		\|\zvec\|^2+2\alpha^2\|\zvec\|^2_{L^2(\Vvec)}\leq&~\| \overline\yvec\|^2, \\
		2\alpha^2\|\zvec\|^2_{L^2(\Vvec)}+\alpha^4\|\zvec\|^2_{L^2(D(\Avec))}
		\leq&~\|\overline \yvec\|^2.
	\end{alignedat}
\]
	From this $\zvec$, we can obtain the unique solution $(\yvec,p)$ to the linear system of the Oseen kind
\[
	\left\{
		\begin{array}{lll}
			\yvec_t  - \Delta \yvec +(\zvec \cdot \nabla) \yvec+ \nabla p = \fvec 			  & \hbox{in}&       Q,     \\
			\nabla\cdot\yvec=0                                      						   	  & \hbox{in}&       Q,      \\
			\yvec  = \ovec         											  	  &\hbox{on}& \Sigma,   \\
			\yvec(0)=\yvec_0                                          					         	  & \hbox{in}&       \Om.
		\end{array}
	\right.
\]
	Since $\fvec\in L^{2}(0,T;\Hvec^{-1}(\Om))$ and $\yvec_0\in \Hvec$, it is clear that
\[
	\begin{array}{c}\dis
		\yvec \in L^2(0, T;\Vvec) \cap C^0([0, T];\Hvec), \quad
		\yvec_t\in L^2(0,T;\Vvec')
	\end{array}
\]
	and the following estimates hold:
\begin{equation}\label{y-ineq}
	\begin{array}{c}
		\noalign{\smallskip}\dis
		\|\yvec\|_{C^0([0,T];\Hvec)}+\|\yvec\|_{L^2(\Vvec)} \leq C_1B_0(\yvec_0,\fvec), \\
		\noalign{\smallskip}\dis
		\|\yvec_t\|_{L^2(\Vvec')} \leq C_2(1+\|\zvec\|_{L^2(D(\Avec))})B_0(\yvec_0,\fvec) \leq C_2(1+\alpha^{-2}\|\overline\yvec\|)B_0(\yvec_0,\fvec).
	\end{array}
\end{equation}

	Now, we introduce the Banach space
\[
	\Wvec = \{\wvec\in L^2(0,T; \Vvec) : \wvec_t\in L^2(0,T;\Vvec')\},
\]
	the closed ball
\[
	\Kvec=\{\, \overline \yvec\in L^2(0,T;\Hvec) : \|\overline\yvec\|\leq C_1\sqrt{T}B_0(\yvec_0,\fvec) \,\}
\]
	and the mapping $\tilde\Lambda_\alpha$, with $\tilde\Lambda_\alpha(\overline \yvec)=\yvec$, for all $\overline \yvec \in L^2(0,T;\Hvec)$.
	Obviously $\tilde\Lambda_\alpha$ is well defined and maps continuously the whole space $L^2(0,T;\Hvec)$ into $\Wvec \cap \Kvec$.

	Notice that any bounded set of $\Wvec$ is relatively compact in the space $L^2(0,T;\Hvec)$, in view of the
	classical results of the Aubin-Lions kind, see for ins\-tance~\cite{Simon}.

	Let us denote by $\Lambda_\alpha$ the restriction to~$\Kvec$ of~$\tilde\Lambda_\alpha$. Then, thanks to \eqref{y-ineq}, $\Lambda_\alpha$ maps
	$\Kvec$ into itself. Moreover, it is clear that $\Lambda_\alpha: \Kvec \mapsto \Kvec$ satisfies the hypotheses of Schauder's Theorem.
	Consequently, $\Lambda_\alpha$ possesses at least one fixed point in $\Kvec$.

	This immediately achieves the proof of the existence of a solution satisfying~\eqref{space-weak}.
	
	The estimates $\eqref{yalpha-ineqq}_{{a}}$, $\eqref{yalpha-ineqq}_{{c}}$ and~$\eqref{yalpha-ineqq}_{{d}}$ are obvious.
	On the other hand,
\[
	\begin{array}{l}\dis
		\|(\yvec_\alpha)_t\|_{L^{\sigma_N}(\Vvec')} \leq C \left( \|\fvec\|_{L^2(\Hvec^{-1})} + \|\yvec_\alpha\|_{L^2(\Vvec)}
		+ \|(\zvec_\alpha\cdot\nabla)\yvec_\alpha\|_{L^{\sigma_N}(\Hvec^{-1})} \right) \\
		\dis \phantom{\|(\yvec_\alpha)_t\|_{L^{\sigma_N}(\Vvec')}}\leq C \left( B_0(\yvec_0,\fvec)
		+ \|\zvec_\alpha\|_{L^{s_N}(\Lvec^4)} \|\yvec_\alpha\|_{L^{s_N}(\Lvec^4)} \right) \\
		\dis \phantom{\|(\yvec_\alpha)_t\|_{L^{\sigma_N}(\Vvec')}} \leq C \left[ B_0(\yvec_0,\fvec)
		+ \left( \|\zvec_\alpha\|_{L^{\infty}(\Hvec)} + \|\zvec_\alpha\|_{L^{2}(\Vvec)} \right)
		\left( \|\yvec_\alpha\|_{L^{\infty}(\Hvec)} + \|\yvec_\alpha\|_{L^{2}(\Vvec)} \right) \right] \\
		\dis \phantom{\|(\yvec_\alpha)_t\|_{L^{\sigma_N}(\Vvec')}}\leq C B_0(\yvec_0,\fvec) (1 + B_0(\yvec_0,\fvec)),
	\end{array}
\]
	where $s_N = 2\,\sigma_N$.
	Here, the third inequality is a consequence of the continuous embedding
\[
	L^\infty(0,T;\Hvec) \cap L^2(0,T;\Vvec) \hookrightarrow L^{s_N}(0,T;\Lvec^4(\Omega)).
\]
	This estimate completes the proof of~\eqref{yalpha-ineqq}.

\

	\textsc{$\bullet$ Uniqueness:}
	Let $(\yvec_\alpha,{p}_\alpha,\zvec_\alpha,{\pi}_{\alpha})$ and $(\yvec'_\alpha,{p'}_\alpha,\zvec'_\alpha,{\pi'}_{\alpha})$ be two solutions
	to~\eqref{Lalpha} and let us introduce $\uvec:=\yvec_\alpha-\yvec'_\alpha$, $q=p_\alpha-p'_\alpha$, $\mvec:=\zvec_\alpha-\zvec'_\alpha$ and
	$h=\pi_\alpha-\pi'_\alpha$. Then
\[
	\left\{
		\begin{array}{lll}
			\uvec_t  -  \Delta \uvec +(\zvec_\alpha \cdot \nabla) \uvec+ \nabla q =-(\mvec \cdot \nabla) \yvec'_\alpha        & \hbox{in} &          Q,    \\
			\mvec-\alpha^2\Delta \mvec +\nabla h=\uvec                                  	             							& \hbox{in} &          Q,    \\
			\nabla \cdot \uvec=0,~\nabla \cdot \mvec= 0                                    					              		 	& \hbox{in} &          Q,    \\
			\uvec = \mvec=\ovec                                                             				                                     		& \hbox{on}&  \Sigma,   \\
			\uvec(0) = \ovec                                                            					                               			& \hbox{in} &  \Om.
		\end{array}
	\right.
\]

	Since $\uvec\in L^\infty(0,T;\Hvec)$, we have $\mvec\in  L^\infty(0,T;D(\Avec))$
	(where the estimate of this norm depends on $\alpha$).
	Therefore, we easily deduce from the first equation of the previous system that
\[
	\frac{1}{2}\frac{d}{dt}\|\uvec\|^2 +\|\nabla\uvec\|^2 \leq \|\mvec\|_\infty\|\nabla\yvec'_\alpha\|\|\uvec\|
\]
	for all $t$. Since $\dis \|\mvec\|_\infty\leq C \|\mvec\|_{D(\Avec)}\leq C\alpha^{-2}\|\uvec\|$, we get
\[
	\frac{1}{2}\frac{d}{dt}\|\uvec\|^2 +\|\nabla\uvec\|^2 \leq C\alpha^{-2}\|\nabla\yvec'_\alpha\|\|\uvec\|^2.
\]
	Therefore, in view of Gronwall's Lemma, we see that $\uvec\equiv0$.
	Accor\-dingly, we also have $\mvec\equiv\ovec$ and uniqueness holds.
\end{proof}

\
	
	We are now going to present some results concerning the existence and uniqueness of a strong solution.
	We start with a global result in the two-dimensional case.
	
\begin{propo}\label{L-S-E-U-L-alpha-2D}
	Assume that $N=2$ and $\alpha>0$ is fixed. Then, for any $\fvec\in L^2(0,T;\Lvec^2(\Om))$ and any $\yvec_0\in \Vvec$,
	there exists exactly one solution $(\yvec_\alpha,p_\alpha, \zvec_\alpha,\pi_\alpha)$ to \eqref{Lalpha}, with
\begin{equation}\label{space-strong}
	\begin{array}{c}
		\noalign{\smallskip}\dis
		\yvec_\alpha\in L^2(0, T;D(\Avec)) \cap C^0([0, T];\Vvec),~(\yvec_\alpha)_t\in L^2(0,T;\Hvec), \\
		\noalign{\smallskip}\dis
		\zvec_\alpha\in L^2(0,T;D(\Avec^2)) \cap C^0([0, T];D(\Avec^{3/2})).
	\end{array}
\end{equation}
	Furthermore, the following estimates hold:
\begin{equation}\label{yalpha-strongg1}
	\begin{alignedat}{2}
		\noalign{\smallskip}\dis
		\|(\yvec_\alpha)_t\|+\|\yvec_\alpha\|_{C^0([0,T];\Vvec)}+\|\yvec_\alpha\|_{L^2(D(\Avec))}\leq&~ B_1(\|\yvec_0\|_\Vvec,\|\fvec\|),\\
		\noalign{\smallskip}\dis
		\|\zvec_\alpha\|^2_{C^0([0,T];\Vvec)}+2\alpha^2\|\zvec_\alpha\|^2_{C^0([0,T];D(\Avec))}\leq&~\| \yvec_\alpha\|^2_{C^0([0,T]; \Vvec)},
	\end{alignedat}
\end{equation}
	where we have introduced the notation
\[
	B_1(r,s):=(r+s)\left[1 + (r+s)^2 \right] e^{C(r^2+s^2)^2}.
\]
\end{propo}

\begin{proof}
	First, thanks to Proposition \ref{G-W-E-U-L-alpha}, we see that there exists a unique weak solution $(\yvec_\alpha,p_\alpha,\zvec_\alpha,\pi_{\alpha})$
	satisfying~\eqref{space-weak}--\eqref{yalpha-ineqq}. In particular, $\zvec_\alpha\in L^2(0,T;\Vvec)$ and we have
\[
\|\zvec_\alpha(t)\| \leq \|\yvec_\alpha(t)\| \quad \text{and} \quad \|\zvec_\alpha(t)\|_{\Vvec} \leq \|\yvec_\alpha(t)\|_{\Vvec}, \quad \forall t\in[0,T].
\]
	
	As usual, we will just check that good estimates can be obtained for $\yvec_\alpha$, $(\yvec_\alpha)_t$ and~$\zvec_\alpha$.
	Thus, we assume that it is possible to multiply by $-\Delta \yvec_\alpha$ the motion equation satisfied by $\yvec_\alpha$.
	Taking into account that $N=2$, we obtain:
\[
	\begin{alignedat}{2}
		\noalign{\smallskip}\dis
		\frac{1}{2}\frac{d}{dt}\|\nabla\yvec_\alpha\|^2+\|\Delta\yvec_\alpha\|^2=&~-(\fvec,\Delta\yvec_\alpha)+((\zvec_\alpha \cdot \nabla)\yvec_\alpha,\Delta\yvec_\alpha)\\
		\noalign{\smallskip}\dis
		\leq&~\|\fvec\|^2+\frac{1}{4}\|\Delta\yvec_\alpha\|^2
		+\|\zvec_\alpha\|^{1/2}\|\zvec_\alpha\|_{\Vvec}^{1/2}\|\yvec_\alpha\|_{\Vvec}^{1/2}\|\Delta\yvec_\alpha\|^{3/2}
\\
		\noalign{\smallskip}\dis
		\leq&~\|\fvec\|^2+\frac{1}{2}\|\Delta\yvec_\alpha\|^2
		+C\|\zvec_\alpha\|^{2}\|\zvec_\alpha\|_{\Vvec}^{2}\|\yvec_\alpha\|_{\Vvec}^{2}.
	\end{alignedat}
\]	
	Therefore,
\[
	\frac{d}{dt}\|\nabla\yvec_\alpha\|^2+\|\Delta\yvec_\alpha\|^2\leq C\left[\|\fvec\|^2+(\|\yvec_\alpha\|^2\|\yvec_\alpha\|^2_{\Vvec})\|\nabla\yvec_\alpha\|^2\right].
\]

	In view of Gronwall's Lemma and the estimates in Proposition \ref{G-W-E-U-L-alpha}, we easily deduce~\eqref{space-strong} and~\eqref{yalpha-strongg1}.
\end{proof}

\

	Notice that, in this two-dimensional case, the strong estimates for $\yvec_\alpha$ in \eqref{yalpha-strongg1} are independent of $\alpha$;
	obviously, we cannot expect the same when $N=3$.
	
	In the three-dimensional case, what we obtain is the following:
	
\begin{propo}\label{G-S-E-U-L-alpha}
	Assume that $N=3$ and $\alpha>0$ is fixed. Then, for any~\,$\fvec\in L^2(0,T;\Lvec^2(\Om))$ and any $\yvec_0\in \Vvec$, there exists
	exactly one solution $(\yvec_\alpha,p_\alpha,\zvec_\alpha,\pi_{\alpha})$ to~\eqref{Lalpha}, with
\[
	\begin{array}{c}
		\noalign{\smallskip}\dis
		\yvec_\alpha\in L^2(0, T;D(\Avec)) \cap C^0([0, T];\Vvec),~(\yvec_\alpha)_t\in L^2(0,T;\Hvec), \\
		\noalign{\smallskip}\dis
		\zvec_\alpha\in L^2(0,T;D(\Avec^2)) \cap C^0([0, T];D(\Avec^{3/2})).
	\end{array}
\]
	Furthermore, the following estimates hold:
\begin{equation}\label{yalpha-strong}
	\begin{alignedat}{2}
		\noalign{\smallskip}\dis
		\|\yvec_\alpha\|_{L^\infty(\Vvec)}+\|\yvec_\alpha\|_{L^2(D(\Avec))}+\|(\yvec_\alpha)_t\|\leq&~B_2(\|\yvec_0\|_\Vvec,\|\fvec\|,\alpha), \\
		\noalign{\smallskip}\dis
		\|\zvec_\alpha\|^2_{L^\infty(\Vvec)}+2\alpha^2\|\zvec_\alpha\|^2_{L^\infty(D(\Avec))}\leq&~\| \yvec_\alpha\|^2_{L^\infty( \Vvec)},
	\end{alignedat}
\end{equation}
	where we have introduced
\[
	B_2(r,s,\alpha):=C(r+s)e^{C\alpha^{-4}(r+s)^2}.
\]
\end{propo}

\begin{proof}
	Thanks to Proposition~\ref{G-W-E-U-L-alpha}, there exists a unique weak solution $(\yvec_\alpha,p_\alpha,\zvec_\alpha,\pi_{\alpha})$ satisfying~\eqref{space-weak}
	and~\eqref{yalpha-ineqq}.

	In particular, we obtain that $\zvec_\alpha\in \Lvec^\infty(Q)$, with
\[
	\dis \|\zvec_\alpha\|_{\infty}\leq \frac{C}{\alpha^2} \left( \|\yvec_0\|_{\Hvec}+\|\fvec\|_{L^2(\Hvec^{-1})} \right).
\]

	On the other hand, $\yvec_0 \in \Vvec$. Hence, from the usual (parabolic) regularity results for Oseen systems, the
	solution to~\eqref{Lalpha} is more regular, i.e.~$\yvec_\alpha\in L^2(0, T;D(\Avec)) \cap C^0([0, T];\Vvec)$ and $(\yvec_\alpha)_t\in L^2(0,T;\Hvec)$.
	Moreover, $\yvec_\alpha$ verifies the first estimate in~\eqref{yalpha-strong}. This achieves the proof.
\end{proof}

	Let us now provide a result concerning three-dimensional strong solutions corresponding to small data, with estimates independent of $\alpha$:
	
\begin{propo}\label{L-S-E-U-L-alpha-3D}
	Assume that $N=3$.
	There exists $C_0 > 0$ such that, for any~$\alpha > 0$, any~$\fvec\in L^\infty(0,T;\Lvec^2(\Om))$ and any~$\yvec_0\in \Vvec$ with
\begin{equation}\label{MMM}
	M:=\max\left\{\|\nabla\yvec_0\|^2,~\|\fvec\|_{L^\infty(\Lvec^2)}^{2/3}\right\}<\frac{1}{\sqrt{2(1+C_0)T}}\,,
\end{equation}
	the Leray-$\alpha$ system \eqref{Lalpha} possesses a unique solution $(\yvec_\alpha,p_\alpha, \zvec_\alpha,\pi_\alpha)$ satisfying
\[
	\begin{array}{c}
		\noalign{\smallskip}\dis
		\yvec_\alpha\in L^2(0, T;D(\Avec)) \cap C^0([0, T];\Vvec),~(\yvec_\alpha)_t\in L^2(0,T;\Hvec), \\
		\noalign{\smallskip}\dis
		\zvec_\alpha\in L^2(0,T;D(\Avec)) \cap C^0([0, T];\Vvec).
	\end{array}
\]
	Furthermore, in that case, the following estimates hold:
\begin{equation}\label{yalpha-strongg}
	\begin{alignedat}{2}
		\noalign{\smallskip}\dis
		\|\yvec_\alpha\|^2_{ C^0([0, T];\Vvec)}+\|\yvec_\alpha\|^2_{L^2(D(\Avec))} \leq &~B_3(M,T),\\
		\noalign{\smallskip}\dis
		\|\zvec_\alpha\|^2_{ C^0([0, T];\Vvec)}+2\alpha^2\|\zvec_\alpha\|^2_{L^2(D(\Avec))}\leq&~\| \yvec_\alpha\|^2_{L^\infty(\Vvec)},
	\end{alignedat}
\end{equation}
	where we have introduced
\[
	B_3(M,T) := 2\left[M^3+M+C_0T\left(\frac{M}{\sqrt{1-2(1+C_0)M^2T}}\right)^3\right].
\]
\end{propo}

\begin{proof}
	The proof is very similar to the proof of the existence of a local in time strong solution to the Navier-Stokes system;
	see for instance~\cite{CONST-FOIAS,TEMAM}.

	As before, there exists a unique weak solution $(\yvec_\alpha,p_\alpha,\zvec_\alpha,\pi_{\alpha})$ and this solution satisfies~\eqref{space-weak} and~\eqref{yalpha-ineqq}.

	By multiplying by $\Delta\yvec_\alpha$ the motion equation satisfied by $\yvec_\alpha$, we see that
\[
	\begin{alignedat}{2}
		\frac{1}{2}\frac{d}{dt}\|\nabla\yvec_\alpha\|^2
					+\|\Delta\yvec_\alpha\|^2=&~(\fvec,\Delta\yvec_\alpha)-((\zvec_\alpha \cdot \nabla) \yvec_\alpha,\Delta\yvec_\alpha)\\
					\leq&~\frac{1}{2}\|\fvec\|^2+\frac{1}{2}\|\Delta\yvec_\alpha\|^2+\|\zvec_\alpha\|_{\Lvec^6} \|\nabla\yvec_\alpha\|_{\Lvec^3} \|\Delta\yvec_\alpha\|\\
					\leq&~\frac{1}{2}\|\fvec\|^2+\frac{1}{2}\|\Delta\yvec_\alpha\|^2+C\|\zvec_\alpha\|_{\Vvec} \|\yvec_\alpha\|_{\Vvec}^{1/2} \|\Delta\yvec_\alpha\|^{3/2}.
	\end{alignedat}
\]	
	Then,
\begin{equation}\label{norm-6}
	\frac{d}{dt}\|\nabla\yvec_\alpha\|^2+\frac{1}{2}\|\Delta\yvec_\alpha\|^2 \leq \|\fvec\|^2 + C_0 \|\nabla\yvec_\alpha\|^6,
\end{equation}
	for some $C_0>0$.

	Let us see that, under the assumption \eqref{MMM}, we have
\begin{equation}\label{nabla-y}
	\dis \|\nabla\yvec_\alpha\|^2 \leq \frac{M}{\sqrt{1-2(1+C_0)M^2T}}, \quad \forall t\in[0,T].
\end{equation}
	Indeed, let us introduce the real-valued function $\psi$ given by
\[
	\dis \psi(t)= \max\left\{ M,\|\nabla\yvec_\alpha(t)\|^2\right\}, \quad \forall t \in [0,T].
\]
	Then, $\psi$ is almost everywhere differentiable and, in view of \eqref{MMM} and \eqref{norm-6}, one has
\[
	{d\psi \over dt} \leq (1+C_0)\psi^3, \quad \psi(0)=M.
\]
	Therefore,
\[
	\psi(t)\leq \frac{M}{\sqrt{1-2(1+C_0)M^2t}}\leq \frac{M}{\sqrt{1-2(1+C_0)M^2T}}
\]
	and, since $\|\nabla\yvec_\alpha\|^2 \leq \psi$, \eqref{nabla-y} holds.
	From this estimate, it is very easy to deduce \eqref{yalpha-strongg}.
\end{proof}

	The following lemma is inspired by a result by Constantin and~Foias for the Navier-Stokes equations, see~\cite{CONST-FOIAS}:
	
\begin{lemma}\label{regularity}
	There exists a continuous function $\phi:\mathbb{R}_+ \mapsto \mathbb{R}_+$, with $\phi(s)\rightarrow 0$ as $s\rightarrow 0^+$,
	satisfying the following properties:

\begin{itemize}
	\item [a)] For $\fvec=\ovec$, any $\yvec_0\in \Hvec$ and any $\alpha > 0$, there exist arbitrarily small times $t^*\in(0,T/2)$ such
	that the corresponding solution to~\eqref{Lalpha} satisfies $\|\yvec_\alpha(t^*)\|^2_{D(\Avec)}\leq \phi(\|\yvec_0\|)$.
	
	\item [b)] The set of these $t^*$ has positive measure.
\end{itemize}

\end{lemma}
\begin{proof}
	We are only going to consider the three-dimensional case;
	the proof in the two-dimensional case is very similar and even easier.
	
	The proof consists of several steps:
	
\begin{itemize}
	\item Let us first see that, for any $k > 3/2$ and any $\tau \in (0,T/2]$, the set
\[
	\dis R_\alpha(k,\tau) := \{\, t\in[0,\tau]: \|\nabla\yvec_\alpha(t)\|^2\leq \frac{k}{\tau}\,\|\yvec_0\|^2 \,\}
\]
	is non-empty and its measure $|R_\alpha(k,\tau)|$ satisfies $\dis |R_\alpha(k,\tau)| \geq \tau/k$.
	
	Obviously, we can assume that $\yvec_0 \not= \ovec$.
	Now, if we suppose that  $\dis |R_\alpha(k,\tau)|< \tau/k$, we have:
\[
         \begin{array}{l} \dis
	\int_0^\tau \|\nabla\yvec_\alpha(t)\|^2\,dt \geq \int_{R_\alpha(k,\tau)^c} \|\nabla\yvec_\alpha(t)\|^2\,dt \geq \left(\tau - {\tau \over k}\right) {k \over \tau} \|\yvec_0\|^2
	\\ \noalign{\smallskip} \dis \phantom{\int_0^\tau \|\nabla\yvec_\alpha(t)\|^2\,dt }
	= (k-1) \|\yvec_0\|^2 > {1 \over 2} \|\yvec_0\|^2.
	\end{array}
\]
	But, since $\fvec=0$ in~\eqref{Lalpha}, we also have the following estimate:
\[
	\int_0^\tau\|\nabla\yvec_\alpha(t)\|^2\,dt \leq {1 \over 2} \|\yvec_\alpha(\tau)\|^2 + \int_0^\tau \|\nabla\yvec_\alpha(t)\|^2\,dt
	= {1 \over 2}\|\yvec_0\|^2.
\]
	So, we get a contradiction and, necessarily, $\dis |R_\alpha(k,\tau)|\geq \tau/k$.
	
	\item Let us choose $\tau \in (0,T/2]$, $k > 3/2$, $t_{0,\alpha} \in R_\alpha(k,\tau)$ and
\[
\overline{T}_\alpha \in \left[ t_{0,\alpha} + {\tau^2 \over 4((1+C_0)k^2 \|\yvec_0\|^4} , t_{0,\alpha} + {3\tau^2 \over 8((1+C_0)k^2 \|\yvec_0\|^4} \right],
\]
	where $C_0$ is the constant furnished by Proposition~\ref{L-S-E-U-L-alpha-3D}.
	Since $\|\nabla\yvec_\alpha(t_{0,\alpha})\|^2\leq \frac{k}{\tau}\,\|\yvec_0\|^2$, there exists exactly one strong solution to~\eqref{Lalpha} in
	~$[t_{0,\alpha},\overline{T}_\alpha]$ starting from $\yvec_\alpha(t_{0,\alpha})$ at time~$t_{0,\alpha}$ and satisfying
	\[
\|\nabla\yvec_\alpha(t)\|^2 \leq {2k\over\tau} \|\yvec_0\|^2, \quad \forall t \in [t_{0,\alpha} ,\overline{T}_\alpha].
	\]
	Obviously, it can be assumed that $\overline{T}_\alpha < T$.

	Let us introduce the set
\[
	\dis G_\alpha(t_{0,\alpha},k,\tau) :=
	\{\, t \in [t_{0,\alpha},\overline{T}_\alpha] :~~\|\Delta\yvec_\alpha(t)\|^2 \leq 65(1+C_0) \left({k \over \tau}\right)^3 \|\yvec_0\|^6 \,\}.
\]
	Then, again $G_\alpha(t_{0,\alpha},k,\tau)$ is non-empty and possesses positive measure.
	More precisely, one has
\begin{equation}\label{contrad-2}
	|G_\alpha(t_{0,\alpha},k,\tau)| \geq { \tau^2 \over 8(1+C_0)k^2\|\yvec_0\|^4 } \,.
\end{equation}
	
	Indeed, otherwise we would get
\[
       \begin{array}{c} \dis
	{1 \over 2} \int_{t_{0,\alpha}}^{\overline{T}_\alpha} \|\Delta\yvec_\alpha(t)\|^2\,dt
	\geq {1 \over 2} \int_{G_\alpha(t_{0,\alpha},k,\tau)^c} \|\Delta\yvec_\alpha(t)\|^2\,dt
	\\ \noalign{\smallskip} \dis
	\geq 65\left(\overline{T}_\alpha \!-\! t_{0,\alpha} \!-\! { \tau^2 \over 8(1\!+\!C_0)k^2\|\yvec_0\|^4 } \right) (1+C_0)
	\!\left({k \over \tau}\right)^3 \!\|\yvec_0\|^6
	\\ \noalign{\smallskip} \dis
	\geq {65 k \over 16 \tau} \|\yvec_0\|^2 > 4 {k \over \tau} \|\yvec_0\|^2.
	\end{array}
\]
	However, arguing as in the proof of Proposition~\ref{L-S-E-U-L-alpha-3D}, we also have
\[
         \begin{array}{l} \dis
	{1 \over 2} \int_{t_{0,\alpha}}^{\overline{T}_\alpha} \|\Delta\yvec_\alpha(t)\|^2\,dt
	\leq \|\nabla\yvec_\alpha(\overline{T}_\alpha)\|^2 + {1 \over 2} \int_{t_{0,\alpha}}^{\overline{T}_\alpha} \|\Delta\yvec_\alpha(t)\|^2\,dt
	\\ \noalign{\smallskip} \dis \qquad
	\leq \|\nabla\yvec_\alpha(t_{0,\alpha})\|^2 + C_0 \int_{t_{0,\alpha}}^{\overline{T}_\alpha} \|\nabla\yvec_\alpha(t)\|^6\,dt
	\\ \noalign{\smallskip} \dis \qquad
	\leq {k \over \tau} \|\yvec_0\|^2 + 8 \left( {k \over \tau} \|\yvec_0\|^2 \right)^3 (\overline{T}_\alpha - t_{0,\alpha}) \leq 4 {k \over \tau} \|\yvec_0\|^2.
	\end{array}
\]
	Consequently, we arrive again to a contradiction and this proves~\eqref{contrad-2}.
	
	\item Let us fix $\tau \in (0,T/2]$ and~$k > 3/2$.
	We can now define $\phi:\mathbb{R}_+\mapsto\mathbb{R}_+$ as follows:
\[
	\dis \phi(s) := 65(1+C_0) {k \over \tau}^3 s^6.
\]
	Then, as a consequence of the previous steps, the set
\[
	\{\,t^*\in [0,T/2] : \|\Avec\yvec_\alpha(t^*)\|^2 \leq \phi(\|\yvec_0\|) \,\}
\]
	is non-empty and it measure is bounded from below by a positive quantity independent of~$\alpha$. This ends the proof.
\end{itemize}
\vspace{-0.9cm}
\end{proof}

	We will end this section with some estimates:
\begin{lemma}\label{interpolation}
	Let $s\in [1,2]$ be given, and let us assume that $\fvec\in \Hvec^{s}(\Om)$.
	Then there exist unique functions $\uvec\in D(\Avec^{s/2})$ and $\pi\in H^{s-1}$
	($\pi$ is unique up to a constant) such that
\begin{equation}\label{eq-lap-stokes}
	\left\{
		\begin{array}{lll}
			\uvec-\alpha^2\Delta \uvec +\nabla \pi=\alpha^2\Delta\fvec				& \hbox{in} &       \Om, \\
			\nabla \cdot \uvec= 0                          	                		   			 	& \hbox{in} &       \Om, \\
			\uvec=\ovec  			                          			    				& \hbox{on}&  \Gamma
		\end{array}
	\right.
\end{equation}
	and there exists a constant $C=C(s,\Om)$ independent of $\alpha$ such that
\begin{equation}\label{1eq-lap-stokes}
	\|\uvec\|_{D(\Avec^{s/2})}\leq C \|\fvec\|_{ \Hvec^{s}(\Om)}.
\end{equation}
	Moreover, by interpolation arguments, $\fvec\in \Hvec^{s}(\Om)$, $s\in (m,m+1)$ then there exist unique functions $\uvec\in D(\Avec^{s/2})$
	and $\pi\in H^{s-1}(\Om)$	($\pi$ is unique up to a constant) which are solution of the problem above and  there exists a constant $C=C(m,\Om)$
	such that
\begin{equation}\label{2eq-lap-stokes}
	\|\uvec\|_{D(\Avec^{s/2})}\leq C \|\fvec\|_{ \Hvec^{s}(\Om)}.
\end{equation}
\end{lemma}
	
    When $s$ is an integer ($s=1$ or $s=2$), the proof can be obtained by adapting the proof of Proposition $2.3$ in \cite{TEMAM}.
	For other values of $s$, it suffices to use a classical interpolation argument (see \cite{TARTAR}).

\
%%%%%%%%%%%%%%%%%%%%%%%%%%%%%%%%%%%%%%%%
%%%%%%%%%%%%%%%%%%%%%%%%%%%%%%%%%%%%%%%%
%%%%%%%% SECTION 2.3
%%%%%%%%%%%%%%%%%%%%%%%%%%%%%%%%%%%%%%%%
%%%%%%%%%%%%%%%%%%%%%%%%%%%%%%%%%%%%%%%%

\subsection{Carleman inequalities and null controllability}

	In this subsection, we will recall some Carleman inequalities and a null controllability result for the Oseen system
\begin{equation}\label{oseen}
	\left\{
		\begin{array}{lll}
     			\yvec_t  - \Delta \yvec +(\hvec \cdot \nabla) \yvec+ \nabla p = \vvec1_\omega	 & \hbox{in}&  Q,         \\
     			\nabla \cdot \yvec = 0     											 & \hbox{in}&  Q,         \\
     			\yvec = \ovec     												 &\hbox{on}&  \Sigma, \\
     			\yvec(0) = \yvec_0												 &\hbox{in}& \Omega,
		\end{array}
	\right.
\end{equation}
	where $\hvec=\hvec(\xvec,t)$ is given.
	The null controllability problem for \eqref{oseen} at time $T>0$ is the following:
	
\begin{quote}{\it
	For any $\yvec_0\in \Hvec$, find $\vvec\in \Lvec^2(\omega\times(0,T))$ such that the associated solution to~\eqref{oseen} satisfies~\eqref{null_condition}.}
\end{quote}
	
	We have the following result from \cite{FC-G-P} (see also \cite{IMANU}):
	
\begin{thm}\label{NC-OSEENN}
	 Assume that $\hvec\in \Lvec^\infty(Q)$ and $\nabla\cdot\hvec = 0$. Then, the linear system \eqref{oseen} is null-controllable at any time $T>0$.
	 More precisely, for each $\yvec_0\in \Hvec$ there exists	 $\vvec\in L^\infty(0,T; \Lvec^2(\omega))$ such that the corresponding solution to \eqref{oseen}
	 satisfies \eqref{null_condition}. Furthermore, the control $\vvec$ can be chosen satisfying the estimate
\begin{equation}\label{contr}
	\|\vvec\|_{L^\infty(\Lvec^2(\om))}\leq e^{K(1+\|\hvec\|^2_\infty)}\|\yvec_0\|,
\end{equation}
	where $K$ only depends on $\Om$, $\om$ and~$T$.
\end{thm}

	The proof is a consequence of an appropriate Carleman inequality for the adjoint system of \eqref{oseen}.

	More precisely, let us consider the backwards in time system
\begin{equation}\label{oseen adjoint}
	\left\{
		\begin{array}{lll}
    			-\mat{\varphi}_t  - \Delta \mat{\varphi} -(\mathbf{h} \cdot \nabla) \mat{\varphi} + \nabla q = \mathbf{G}     & \hbox{in}&  Q,        \\
     			\nabla \cdot \mat{\varphi} = 0    													         & \hbox{in}&  Q,         \\
     			\mat{\varphi} = \textbf{0} 															         & \hbox{on}& \Sigma, \\
    			\mat{\varphi}(T) = \mat{\varphi}_0, 													         & \hbox{in}&  \Omega.
		\end{array}
	\right.
 \end{equation}

	The following result is established in \cite{FC-G-P}:
\begin{propo}\label{Carleman}
	Assume that $\hvec\in \Lvec^{\infty}(Q)$ and $\nabla\cdot\hvec=0$.
	There exist positive continuous functions $\alpha$, $\alpha^*$, $\hat\alpha$,  $\xi$, $\xi^*$ and $\hat\xi$ and positive
	constants $ \hat{s}$, $\hat{\lambda}$ and~$\widehat{C}$, only depending on $\Omega$ and $\omega$, such that, for any
	$\mat{\varphi} _0\in \Hvec$ and any ${\mathbf{G}}\in\Lvec^2(Q)$, the solution to the adjoint system \eqref{oseen adjoint} satisfies:
\begin{equation}\label{CAR}
	\begin{array}{c}
		\dis
		\iint_Q e^{-2s\alpha} \left[s^{-1}\xi^{-1}(|\mat{\varphi}_{t}|^2+|\Delta\mat{\varphi}|^2)+s\xi\lambda^2
		|\nabla\mat{\varphi}|^2+s^3\xi^3\lambda^4 |\mat{\varphi}|^2\right]d\xvec\,dt\\
		\dis
		\leq\widehat{C}(1+T^2)\Biggl(s^{15/2}\lambda^{20}\displaystyle\iint_Qe^{-4s\hat{\alpha}+2s\alpha^*}{\xi^*}^{15/2}|\mathbf{{G}}|^2\,d\mathbf{x}\,dt  \Bigr.\\
		\dis
		+\Bigl. s^{16}\lambda^{40}\displaystyle\iint_{\omega\times(0,T)} e^{-8s\hat{\alpha}+6s\alpha^*}{\xi^*}^{16}|\mat{\varphi}|^2\,d\mathbf{x}\,dt \Biggr),
	\end{array}
\end{equation}
	for all  $s \geq  \hat{s}(T^4 + T^8) $ and for all $\lambda \geq \hat{\lambda}\Bigl(1 + \|\mathbf{h}\|_{\infty}+ e^{\hat{\lambda}T \|\mathbf{h}\|^2_{\infty}}\Bigl)$.
\end{propo}

	 Now, we are going to construct the a null-control for \eqref{oseen} like in \cite{FC-G-P}. First, let us introduce the auxiliary extremal problem
\begin{equation}\label{optm}
	\left\{
		\begin{array}{l}
			\noalign{\smallskip}
			\text{Minimize } \ \displaystyle{\frac{1}{2}\left\{\iint_Q\hat\rho^2|\yvec|^2\,d\xvec\,dt+\iint_{\omega\times(0,T)}\hat\rho_0^2|\vvec|^2\,d\xvec\,dt\right\}}\\
			\noalign{\smallskip}	
			\text{Subject to}~(\yvec,\vvec) \in  \mathcal{M}(\yvec_0,T),
		\end{array}
	\right.
\end{equation}
	where the linear manifold $\dis\mathcal{M}(\yvec_0,T)$ is given by
\[
	\dis\mathcal{M}(\yvec_0,T) = \{\, (\yvec,\vvec) : \vvec \in \Lvec^2(\omega \times (0,T)), \ (\yvec,p)~\hbox{solves \eqref{oseen}} \,\}
\]
	and $\hat\rho$, $\hat\rho_0$ are respectively given by
\[
	\hat\rho=s^{-15/4}\lambda^{-10}e^{2s\hat{\alpha}-s\alpha^*}{\xi^*}^{-15/4}, \quad \hat\rho_0=s^{-8}\lambda^{-20} e^{4s\hat{\alpha}-3s\alpha^*}{\xi^*}^{-8}.
\]
	
	It can be proved that  \eqref{optm} possesses exactly one solution $(\yvec,\vvec)$ satisfying
\[
	\|\vvec\|_{L^2(\Lvec^2(\om))}\leq  e^{K(1+\|\hvec\|^2_\infty)}\|\yvec_0\|,
\]
	where $K$ only depends on $\Om$, $\om$ and $T$.

	Moreover, thanks to the Euler-Lagrange characterization, the solution to the extremal problem \eqref{optm} is given by
\[
	\yvec=\hat\rho^{-2}(-\mat{\varphi}_t  - \Delta \mat{\varphi} -(\mathbf{h} \cdot \nabla) \mat{\varphi} + \nabla q)
	~~\hbox{and}~~\vvec=-\hat\rho_0^{-2}\mat{\varphi}1_{\om}.
\]

	From the Carleman inequality~\eqref{CAR}, we can conclude that $\rho_2^{-1}\mat{\varphi}\in L^\infty(0,T;\Lvec^2(\Om))$ and
\[
	\|\rho_2^{-1}\mat{\varphi}\|_{L^\infty(\Lvec^2)}\leq C\|\hat\rho_0^{-1}\mat{\varphi}\|_{L^2(\Lvec^2(\om))},
\]
	where $\rho_2=s^{1/2}\xi^{1/2}e^{s\alpha}$.\\
	Hence,
\[
	\vvec=-(\hat\rho_0)^{-2}\mat{\varphi}1_{\om}=-(\hat\rho_0^{-2}\rho_2)(\rho_2^{-1}\mat{\varphi}1_{\om})\in L^\infty(0,T;\Lvec^2(\Om))
\]
	and, therefore,
\[
	\|\vvec\|_{L^\infty(\Lvec^2(\om))}\leq C \|\vvec\|_{L^2(\Lvec^2(\om))}\leq e^{K(1+\|\hvec\|^2_\infty)}\|\yvec_0\|.
\]

%%%%%%%%%%%%%%%%%%%%%%%%%%%%%%%%%%%%%%%%
%%%%%%%%%%%%%%%%%%%%%%%%%%%%%%%%%%%%%%%%
%%%%%%%% SECTION 3    %%%%%%%%%%%%%%%%%%%%%%%%
%%%%%%%%%%%%%%%%%%%%%%%%%%%%%%%%%%%%%%%%
%%%%%%%%%%%%%%%%%%%%%%%%%%%%%%%%%%%%%%%%

%\section{Proofs of the main results}\label{Sec3}
\section{The distributed case: Theorems \ref{NC-LERAY} and \ref{CONVERGENCE}}\label{Sec3}

	This section is devoted to prove the local null controllability of~\eqref{CLalpha} and the uniform controllability property in~Theorem~\ref{CONVERGENCE}.
	\\
	\\
\noindent
{\bf Proof of Theorem~\ref{NC-LERAY}:}
	We will use a fixed point argument.
	Contrarily to the case of the Navier-Stokes equations, it is not sufficient to work here with controls in $\Lvec^2(\om\times(0,T))$.
	Indeed, we need a space $\mathbf{Y}$ for $\yvec$ that ensures $\zvec$ in $\Lvec^\infty(Q)$ and a space $\mathbf{X}$ for $\vvec$
	guaranteeing that the solution to \eqref{oseen} with $\hvec=\zvec$ belongs to a compact set of $\mathbf{Y}$. Furthermore,
	we want estimates in $\mathbf{Y}$ and $\mathbf{X}$ independent of $\alpha$.
	
	In view of Lemma~\ref{regularity}, in order to prove Theorem~\ref{NC-LERAY}, we just need to consider the case in which the initial state $\yvec_0$
	belongs to $D(\Avec)$ and possesses a sufficiently small norm in $D(\Avec)$.

	Let us fix $\sigma$ with $N/4<\sigma<1$.
	Then, for each $\tilde\yvec\in L^\infty(0, T;D(\Avec^{\sigma}))$, let $(\zvec,\pi)$ be the unique solution to
\[
	\left\{
		\begin{array}{lll}
			\zvec-\alpha^2\Delta \zvec +\nabla \pi=\tilde\yvec                         & \hbox{in} &       Q,    \\
			\nabla \cdot \zvec= 0                          	                		            & \hbox{in} &       Q,     \\
			\zvec=\textbf{0}                                                       			    & \hbox{on}&  \Sigma.
		\end{array}
	\right.
\]
	Since $\tilde\yvec\in L^\infty(0, T;D(\Avec^{\sigma}))$, it is clear that $\zvec\in L^\infty(0, T;D(\Avec^{\sigma}))$.
	Then, thanks to Theorem \ref{embb}, we have $\zvec\in \Lvec^\infty(Q)$ and the following is satisfied:
\begin{equation}\label{zalpha-ineqq}
	\begin{alignedat}{2}
		\|\mathbf{z}\|_{L^\infty(0, T;D(\Avec^{\sigma}))}^2 + 2\alpha^2 \|\mathbf{z}\|^2_{L^\infty(D(\Avec^{1/2+\sigma}))}
		\leq&~\|\tilde\yvec\|_{L^\infty(0, T;D(\Avec^{\sigma}))}^2, \\
		2\alpha^2\|\mathbf{z}\|^2_{L^\infty(D(\Avec^{1/2+\sigma}))}+\alpha^4\|\zvec\|^2_{L^\infty(D(\Avec^{1+\sigma}))}
		\leq&~\|\tilde\yvec\|_{L^\infty(0, T;D(\Avec^{\sigma}))}^2.
	\end{alignedat}
\end{equation}
	In particular, we have:
\[
	\|\mathbf{z}\|_{L^\infty(0, T;D(\Avec^{\sigma}))}	\leq\|\tilde{\mathbf{y}}\|_{L^\infty(0, T;D(\Avec^{\sigma}))}.
\]

	Let us consider the system \eqref{oseen} with $\hvec$ replaced by $\zvec$.
	In view of Theorem \ref{NC-OSEENN}, we can associate to $\zvec$ the null control $\vvec$ of minimal norm in
	$L^\infty(0,T;\Lvec^2(\omega))$ and the corresponding solution $(\yvec,p)$ to \eqref{oseen}.

	Since $\yvec_0\in D(\Avec)$, $\zvec\in \Lvec^\infty(Q)$ and $\vvec\in L^{\infty}(0,T;\Lvec^2(\omega))$, we have
\[
	\begin{array}{c}\dis
		\yvec \in L^2(0, T;D(\Avec)) \cap C^0([0, T];\Vvec),~\dis \yvec_t \in L^2(0,T;\Hvec)
	\end{array}
\]
	and the following estimate holds:
\begin{equation}\label{yy-ineqq}
	\|\yvec_t\|^2_{L^2(\Hvec)}+\|\yvec\|^2_{L^2(D(\Avec)) }+\|\mathbf{y}\|^2_{L^\infty(\Vvec)}
	\leq C(\|\mathbf{y}_0\|^2_{\Vvec}+\|\vvec\|^2_{L^\infty(\Lvec^{2}(\om))})e^{C\|\zvec\|^2_{\infty}}.
\end{equation}

	We will use the following result:
\begin{lemma}\label{Regularity}
	One has $\yvec\in L^\infty(0, T;D(\Avec^{\sigma'}))$, for all $\sigma'\in(\sigma,1)$, with
\[
\begin{alignedat}{2}
	\|\yvec\|_{L^\infty(D(\Avec^{\sigma'}))} \leq  C(\|\yvec_0\|_{D(\Avec)}+\|\vvec\|_{L^\infty(\Lvec^2(\om))})e^{C\|\tilde{\mathbf{y}}\|^2_
	{L^\infty(D(\Avec^{\sigma}))}}.
\end{alignedat}
\]
\end{lemma}

\begin{proof}
	In view of \eqref{oseen}, $\yvec$ solves the following abstract initial value problem:
\[
	\left\{
		\begin{array}{lll}
			\yvec_t = - \Avec\yvec - \Pvec((\zvec\cdot \nabla)\yvec) + \Pvec(\vvec1_\omega)                      & \hbox{in}&       [0,T], \\
			\yvec(0) = \yvec_0.
		\end{array}
	\right.
\]
	This system can be rewritten as the nonlinear integral equation
\[
	\yvec(t)=e^{-t\Avec}\yvec_0-\int_0^t e^{-(t-s)\Avec}\Pvec((\zvec\cdot\nabla)\yvec)(s)~ds+\int_0^te^{-(t-s)\Avec}\Pvec(\vvec1_\omega)(s)\,ds.
\]
\vspace{-0.1cm}
	Consequently, applying the  operator $\Avec^{\sigma'}$ to both sides, we have
\[
	\Avec^{\sigma'}\yvec(t)=\Avec^{\sigma'}e^{-t\Avec}\yvec_0+\int_0^t \Avec^{\sigma'}e^{-(t-s)\Avec}\left[-P((\zvec\cdot\nabla)\yvec)(s)+P(\vvec1_\omega)(s)\right]\,ds.
\]
\vspace{-0.2cm}
	Taking norms in both sides and using Theorem \ref{estimatesemigroup}, we see that
\vspace{-0.1cm}
\[
	\begin{alignedat}{2}
		\|\Avec^{\sigma'}\yvec\|(t)  &\leq \|\yvec_0\|_{D(A^{\sigma'})}+\int_0^t(t-s)^{-\sigma'}
		\left[\|\zvec(s)\|_\infty \|\nabla\yvec(s)\|+\|\vvec(s)1_\omega\| \right]\,ds\\
						         &\leq  C\|\yvec_0\|_{D(\Avec)}+(\|\zvec\|_\infty \|\yvec\|_{L^\infty(\Vvec)}+\|\vvec\|_{L^\infty(\Lvec^2(\om))})\int_0^t(t-s)^{-\sigma'}\,ds.
	\end{alignedat}
\]
	Now, using \eqref{zalpha-ineqq} and \eqref{yy-ineqq} and taking into account that $\sigma'<1$, we easily obtain that
\[
		\|\Avec^{\sigma'}\yvec\|(t)\leq  C(\|\yvec_0\|_{D(\Avec)}+\|\vvec\|_{L^\infty(\Lvec^2(\om))})
		\left[1+\|\tilde{\mathbf{y}}\|_{L^\infty(D(\Avec^{\sigma}))}	e^{C\|\tilde{\mathbf{y}}\|^2_{L^\infty(D(\Avec^{\sigma}))}}\right].
\]
	This ends the proof.
\end{proof}

	Now, let us set
\[
	\Wvec = \{\, \wvec\in L^\infty(0,T; D(\Avec^{\sigma'})) : \wvec_t\in L^2(0,T;\Hvec) \,\}
\]
	and let us consider the closed ball
	$$
	\Kvec=\{\, \tilde \yvec\in L^\infty(0, T;D(\Avec^{\sigma})) : \|\tilde{\mathbf{y}}\|_{L^\infty(D(\Avec^{\sigma}))}\leq 1\,\}
	$$
	and the mapping $\tilde\Lambda_\alpha$, with $\tilde\Lambda_\alpha(\tilde \yvec)=\yvec$ for all $\tilde \yvec \in  L^\infty(0, T;D(\Avec^{\sigma}))$.
	Obviously, $\tilde\Lambda_\alpha$ is well defined; furthermore, in view of Lemma~\ref{Regularity} and \eqref{yy-ineqq}, it maps the whole space
	$ L^\infty(0, T;D(\Avec^{\sigma}))$ into $\Wvec$.

	Notice that, if $\mathbf{U}$ is bounded set of $\Wvec$ then it is relatively compact in the space $ L^\infty(0, T;D(\Avec^{\sigma}))$, in view of the
	classical results of the Aubin-Lions kind, see for instance~\cite{Simon}.

	Let us denote by $\Lambda_\alpha$ the restriction to~$\Kvec$ of~$\tilde\Lambda_\alpha$.
	Then, thanks to Lemma~\ref{Regularity} and \eqref{contr}, if $\|\yvec_0\|_{D(\Avec)}\leq \varepsilon$ (independent of $\alpha$!) $\Lambda_\alpha$ maps $\Kvec$ into itself.
	Moreover, it is clear that $\Lambda_\alpha: \Kvec \mapsto \Kvec$ satisfies the hypotheses of Schauder's Theorem.
	Indeed, this nonlinear mapping is continuous and compact
	(the latter is a consequence of the fact that, if $\mathbf{B}$ is bounded in $ L^\infty(0, T;D(\Avec^{\sigma}))$, then $\tilde\Lambda_\alpha(\mathbf{B})$
	is bounded in $\Wvec$). Consequently, $\Lambda_\alpha$ possesses at least one fixed point in $\Kvec$, and this ends the proof of Theorem \ref{NC-LERAY}.
	\Fin

\

\noindent
{\bf Proof of Theorem~\ref{CONVERGENCE}:}
 	Let $\vvec_\alpha$ be a null control for \eqref{CLalpha} satisfying \eqref{v-unif} and let $(\yvec_\alpha,p_\alpha,\zvec_\alpha,\pi_\alpha)$ be the state associated to $\vvec_\alpha$.
	From \eqref{v-unif} and the estimates \eqref{yalpha-ineqq} for the solutions $\yvec_\alpha$, there exist $\vvec\in  L^\infty\left(0, T;\Lvec^2(\om)\right)$ and $\yvec\in L^\infty(0,T;\Hvec)\cap L^2(0,T;\Vvec)$
	with $\mathbf{y}_t\in L^{\sigma_N}(0,T;\Vvec')$ such that, at least for a subsequence
\[
	\begin{alignedat}{2}
			\noalign{\smallskip}   \mathbf{v}_\alpha      &\rightarrow \mathbf{v}~\hbox{weakly-$\star$ in  } L^\infty\left(0, T;\Lvec^2(\om)\right), \\
			\noalign{\smallskip}   \mathbf{y}_\alpha      &\rightarrow \mathbf{y}~\hbox{weakly-$\star$ in } L^\infty\left(0, T;\Hvec\right) \hbox{ and weakly in } L^2\left(0, T;\Vvec\right), \\
			\noalign{\smallskip}	(\mathbf{y}_\alpha)_t &\rightarrow \mathbf{y}_t~\hbox{weakly in } L^{\sigma_N}(0,T;\Vvec').
	\end{alignedat}
\]

	Since $\Wvec:=\{\mvec\in L^2\left(0, T;\Vvec\right):~ \mvec_t\in L^{\sigma_N}(0,T;\Vvec')\}$ is continuously
	and compactly embedded in $\Lvec^2(Q)$, we have that
\[
	\mathbf{y}_\alpha \to \mathbf{y} \hbox{ in }\Lvec^2(Q) \hbox{ and a.e.}
\]
	This is sufficient to pass to the limit in the equations satisfied by $\yvec_\alpha$, $\vvec_\alpha$ and $\zvec_\alpha$. We conclude that
	$\yvec$ is, together with some pressure $p$, a solution to the Navier-Stokes equations associated to a control $\vvec$ and satisfies
	\eqref{null_condition}.
	\Fin

%%%%%%%%%%%%%%%%%%%%%%%%%%%%%%%%%%%%%%%%
%%%%%%%%%%%%%%%%%%%%%%%%%%%%%%%%%%%%%%%%
%%%%%%%% SECTION 4
%%%%%%%%%%%%%%%%%%%%%%%%%%%%%%%%%%%%%%%%
%%%%%%%%%%%%%%%%%%%%%%%%%%%%%%%%%%%%%%%%

\section{The boundary case: Theorems \ref{NC-LERAY-BOUNDARY} and \ref{CONVERGENCE-BOUNDARY}}\label{Sec4}

	This section is devoted to prove the local boundary null controllability of~\eqref{BBCLalpha} and the uniform controllability property
	in~Theorem~\ref{CONVERGENCE-BOUNDARY}.
	\\
	\\
\noindent
{\bf Proof of Theorem~\ref{NC-LERAY-BOUNDARY}:}
	Again, we will use a fixed point argument. Contrarily to the case of distributed controllability, we will have to work in
	a space $\mathbf{\tilde Y}$ of functions defined in an extended domain.
	
	Let $\widetilde \Om$ be given, with $\Om\subset\widetilde \Om$ and $\partial\widetilde \Om \cap \Gamma = \Gamma \setminus \gamma$ such that $\partial\widetilde \Om$ is of class~$C^2$ (see~Fig.~\ref{fig1}).
	Let $\om\subset \widetilde \Om\setminus\overline\Om$ be a non-empty open subset and let us introduce $\widetilde Q:= \tilde \Om\times(0,T)$ and $\widetilde \Sigma:=  \partial\widetilde \Om\times(0,T)$.
	The spaces and operators associate to the domain $\widetilde \Om$ will be denoted by $\widetilde \Hvec$, $\widetilde \Vvec$, $\widetilde\Avec$, etc. 

\begin{rmq}\label{reg-data} \rm
	In view of Lemma~\ref{regularity}, for the proof of Theorem~\ref{NC-LERAY-BOUNDARY} we just need to consider the case in which the initial
	state $\yvec_0$ belongs to $\Vvec$ and possesses a sufficiently small norm in $\Vvec$.
	Indeed, we only have to take initially $\hvec_\alpha\equiv \ovec$ and apply Lemma~\ref{regularity} to the solution to~\eqref{BBCLalpha}. \Fin
\end{rmq}	

	Let $\yvec_0 \in \Vvec$ be given and let us introduce the extension by zero 
	$\tilde \yvec_0$ of~$\yvec_0$. Then $\tilde \yvec_0\in \widetilde \Vvec$.

\begin{figure}[h]
\centering\includegraphics[scale=0.45]{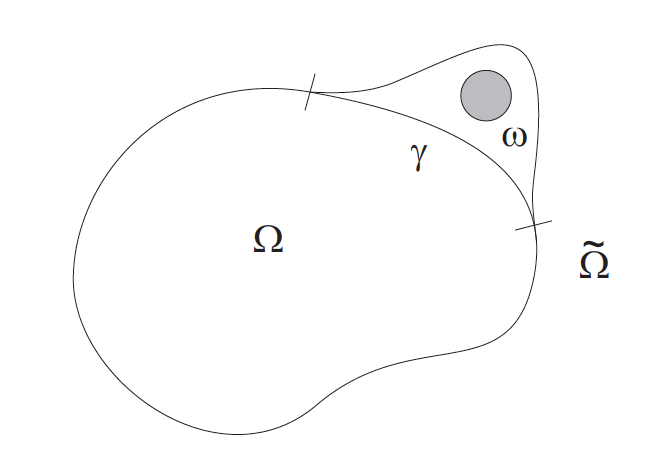}
\caption{The domain $\widetilde\Om$}
\label{fig1}
\end{figure}

	We will use the following result, similar to Lemma~\ref{regularity}, whose proof is postponed to the end of the section:
\begin{lemma}\label{regularity_BC}
	There exists a continuous function $\phi:\mathbb{R}_{+}\mapsto \mathbb{R}_{+}$
	satisfying $\phi(s)\rightarrow 0$ as $s\rightarrow 0^+$ with the following property:
	
\begin{itemize}

	\item [a)]
	For any $\yvec_0\in \Vvec$ and any $\alpha > 0$, there exist times $T_0 \in (0,T)$, controls~$\hvec_\alpha\in L^2(0,T_0;\Hvec^{1/2}(\Gamma))$ with $\int_{\gamma} \hvec_\alpha\cdot \nvec \,d\Gamma \equiv 0$, associated solutions $(\yvec_\alpha,p_\alpha,\zvec_\alpha,\pi_\alpha)$ to \eqref{BBCLalpha} in~$\Om \times (0,T_0)$ and arbitrarily small times $t^*\in(0,T/2)$ such that the $\yvec_\alpha$ can be extended to $\widetilde \Om \times (0,T_0)$ and the extensions satisfy $\|\tilde\yvec_\alpha(t^*)\|^2_{D(\widetilde\Avec)}\leq \phi(\|\yvec_0\|_{\Vvec})$.
	
	\item [b)] The set of these $t^*$ has positive measure.
	
	\item [c)] The controls $\hvec_\alpha$  are uniformly bounded, i.e.
	\[
		\|\hvec_\alpha\|_{L^\infty(0,T_0;\Hvec^{1/2}(\gamma))}\leq C.
	\]
\end{itemize}
\end{lemma}	

	In view of Lemma~\ref{regularity_BC}, for the proof of Theorem~\ref{NC-LERAY-BOUNDARY}, we just need to consider the case in which
	the initial state $\yvec_0$ is such that its extension $\tilde \yvec_0$ to $\tilde \Om$ belongs to $D(\tilde\Avec)$ and possesses a
	sufficiently small norm in $D(\tilde\Avec)$.

	We will prove that there exists
	$(\tilde\yvec_\alpha,\tilde p_\alpha,\zvec_\alpha,\pi_\alpha,\tilde\vvec)$, with $\tilde \vvec \in L^\infty(0,T;\Lvec^2(\om))$, satisfying
\begin{equation}\label{final-bound}
		\left\{
			\begin{array}{lll}
				\tilde\yvec_t  -  \Delta \tilde\yvec +(\tilde\zvec \cdot \nabla) \tilde\yvec+ \nabla\tilde p = \tilde\vvec1_{\om}
				& \hbox{in} &     \widetilde Q,  \\
				\zvec-\alpha^2\Delta \zvec +\nabla \pi=\tilde\yvec                                            				
				& \hbox{in} &       Q,  \\
				\nabla \cdot \tilde\yvec=0         		                                        									
				& \hbox{in} &       \widetilde Q, \\
				\nabla \cdot \zvec=0            		                                        									
				& \hbox{in} &        Q, \\
				\tilde\yvec = \ovec                                                                               							
				& \hbox{on}&       \widetilde\Sigma, \\
				\zvec =\tilde \yvec                                                                              								
				& \hbox{on}&       \Sigma, \\
				\tilde\yvec(0) = \tilde\yvec_0                                                                                       				
				& \hbox{in} &       \widetilde \Om
			\end{array}
		\right.
\end{equation}
	and~$\tilde \yvec(T) =\ovec$ in~$\widetilde \Om$, where $\tilde\zvec$ is the extension by zero of $\zvec$. Obviously, if this were the case, the restriction $(\yvec,p)$ of~$(\tilde\yvec,\tilde p)$ to $Q$, the couple $(\zvec,\pi)$ and the lateral trace $\hvec:= \tilde \yvec|_{\gamma\times(0,T)}$ would satisfy~\eqref{BBCLalpha} and~\eqref{null_condition}.

	Let us fix $\sigma$ with $N/4<\sigma<\beta<1$.
	Then, for each $\overline\yvec\in L^\infty(0, T;D(\tilde\Avec^{\sigma}))$, let $\wvec = \wvec(\xvec,t)$ and $\pi=\pi(\xvec,t)$ be the unique solution to
\[
	\left\{
		\begin{array}{lll}
			\wvec-\alpha^2\Delta \wvec +\nabla \pi=\alpha^2\Delta\overline\yvec			& \hbox{in} &       Q,    \\
			\nabla \cdot \wvec= 0                          	                		   			 	& \hbox{in} &       Q,     \\
			\wvec=\ovec  			                          			    				& \hbox{on}&  \Sigma.
		\end{array}
	\right.
\]

	Since $\overline\yvec\in L^\infty(0, T;D(\tilde\Avec^{\sigma}))$, its restriction to $Q$ belongs to $L^\infty(0, T;\Hvec^{2\sigma}(\Om))$.
	Then, Lemma \ref{interpolation} implies $\wvec\in L^\infty(0, T;D(\Avec^{\sigma}))$ and, thanks to Theorem \ref{embb}, we also have
	$\wvec\in \Lvec^\infty(Q)$ and
\[
	\begin{alignedat}{2}
		\|\wvec\|_{L^\infty(0, T;D(\Avec^{\sigma}))}^2 \leq&~C\|\overline\yvec\|_{L^\infty(0, T;D(\tilde\Avec^{\sigma}))}^2,
	\end{alignedat}
\]
	where $C$ is independent of $\alpha$.
	
	Let $\tilde\wvec$ be the extension by zero of $\wvec$  and let us set $\tilde \zvec: = \overline \yvec+ \tilde\wvec$. Let us consider the system
	\eqref{oseen} with $\hvec$ replaced by $\tilde\zvec$ and $\Om$ replaced by $\widetilde \Om$. In view of Theorem \ref{NC-OSEENN}, we can
	associate to $\tilde\zvec$ the null control $\tilde\vvec$ of minimal norm in	$L^\infty(0,T;\Lvec^2(\tilde\omega))$ and the corresponding solution
	$(\tilde\yvec,\tilde p)$ to \eqref{oseen}. Since $\tilde\yvec_0\in D(\tilde\Avec)$, $\tilde\zvec\in \Lvec^\infty(\widetilde Q)$ and
	$\tilde\vvec\in L^{\infty}(0,T;\Lvec^2(\tilde\omega))$, we have
\[
	\begin{array}{c}\dis
		\tilde\yvec \in L^2(0, T;D(\tilde\Avec)) \cap C^0([0, T];\tilde\Vvec),~\dis \tilde\yvec_t \in L^2(0,T;\tilde\Hvec)
	\end{array}
\]
	and the following estimate holds:
\begin{equation}\label{tilyy-ineqq}
	\|\tilde \yvec_t\|^2_{L^2(\tilde\Hvec)}+\|\tilde\yvec\|^2_{L^2(D(\tilde\Avec)) }+\|\tilde\yvec\|^2_{L^\infty(\tilde\Vvec)}
	\leq C(\|\tilde\yvec_0\|^2_{\tilde\Vvec}+\|\tilde\vvec\|^2_{L^\infty(\Lvec^{2}(\tilde\om))})e^{C\|\tilde\zvec\|^2_{\infty}}.
\end{equation}

	Also, in account of Lemma \ref{Regularity}, one has $\tilde\yvec\in L^\infty(0, T;D(\tilde\Avec^{\beta}))$ and
\[
\begin{alignedat}{2}
	\|\tilde\yvec\|_{L^\infty(D(\tilde\Avec^{\beta}))} \leq  C(\|\tilde\yvec_0\|_{D(\tilde\Avec)}+\|\tilde\vvec\|_{L^\infty(\Lvec^2(\tilde\om))})e^{C\|\overline{\mathbf{y}}\|_{L^\infty(0, T;D(\tilde\Avec^{\sigma}))}}).
\end{alignedat}
\]

	Now, let us set
\[
	\Wvec = \{\, \mvec\in L^\infty(0,T; D(\tilde\Avec^{\beta})) : \mvec_t\in L^2(0,T;\tilde\Hvec) \,\},
\]
	and let us consider the closed ball
	$$
	\Kvec=\{\, \overline \yvec\in L^\infty(0, T;D(\tilde\Avec^{\sigma})) : \|\overline{\mathbf{y}}\|_{L^\infty(D(\tilde\Avec^{\sigma}))}\leq 1\,\}
	$$
	and the mapping $\tilde\Lambda_\alpha$, with $\tilde\Lambda_\alpha(\overline\yvec)=\tilde\yvec$ for all $\overline\yvec \in
	L^\infty(0, T;D(\tilde\Avec^{\sigma}))$.
	Obviously, $\tilde\Lambda_\alpha$ is well defined and maps the whole space
	$ L^\infty(0, T;D(\tilde\Avec^{\sigma}))$ into $\Wvec$. Furthermore, any bounded set $\Uvec\subset\Wvec$ then it is relatively compact
	in $L^\infty(0, T;D(\tilde\Avec^{\sigma}))$.
	
	Let us denote by $\Lambda_\alpha$ the restriction to~$\Kvec$ of~$\tilde\Lambda_\alpha$.
	Thanks to Lemma~\ref{Regularity} and \eqref{contr}, there exists $\eps>0$ (independent of $\alpha$) such that if $\|\tilde\yvec_0\|_{D(\tilde\Avec)}
	\leq \varepsilon$, $\Lambda_\alpha$ maps $\Kvec$ into itself and it is clear that $\Lambda_\alpha: \Kvec \mapsto \Kvec$ satisfies the hypotheses of
	Schauder's Theorem. Consequently, $\Lambda_\alpha$ possesses at least one fixed point in $\Kvec$ and \eqref{final-bound} possesses a
	solution.This ends the proof of Theorem \ref{NC-LERAY-BOUNDARY}.
\Fin
	\\
	\\
\begin{proof}[Proof of Theorem~\ref{CONVERGENCE-BOUNDARY}]
	The proof is easy, in view of the previous uniform estimates. It suffices to adapt the argument in the proof of Theorem \ref{CONVERGENCE}
	and deduce the existence of subsequences that converge (in an appropriate sense) to a solution to \eqref{BBC-NS} satisfying \eqref{null_condition}. For brevity, we omit the details.
\end{proof}
	\\
	\\
\begin{proof}[Proof of Lemma~\ref{regularity_BC}]
	For instance, let us only consider the case $N=3$.
	We will reduce the proof to the search of a fixed point of another mapping $\Phi_\alpha$.

	For any~$\yvec_0 \in \Vvec$, any~$T_0 \in (0,T)$ and~any $\overline\yvec\in L^4(0,T_0;\widetilde\Vvec))$, let $(\wvec,\pi)$ be the unique solution to
\[
	\left\{
		\begin{array}{lll}
			\wvec-\alpha^2\Delta \wvec +\nabla \pi=\alpha^2\Delta\overline\yvec			& \hbox{in} &  \Om \times (0,T_0),    \\
			\nabla \cdot \wvec= 0                          	                		   			 	& \hbox{in} &  \Om \times (0,T_0),     \\
			\wvec=\ovec  			                          			    				& \hbox{on}&  \Gamma \times (0,T_0),
		\end{array}
	\right.
\]
	let $\tilde\wvec$ be the extension by zero of $\wvec$, let us set $\tilde \zvec: = \overline \yvec+ \tilde\wvec$ and let us introduce the Oseen system
\[
	\left\{
		\begin{array}{lll}
			\tilde \yvec_t  - \Delta \tilde \yvec +(\tilde \zvec \cdot \nabla) \tilde \yvec+ \nabla \tilde p = \ovec  & \hbox{in}&  \widetilde\Om\times (0,T_0),     \\
			\nabla\cdot\tilde \yvec=0                                      						   	  		& \hbox{in}& \widetilde\Om\times (0,T_0),   \\
			\tilde \yvec  = \ovec         											  	 		&\hbox{on}& \partial\widetilde\Om\times (0,T_0),  \\
			\tilde \yvec(0)=\tilde \yvec_0                                          					         	& \hbox{in}&      \widetilde\Om.
		\end{array}
	\right.
\]

	It is clear that the restriction of $\overline{\yvec}$ to $\Om\times(0,T_0)$ belongs to $L^4(0, T_0;\Hvec^{1}(\Om))$, whence we have from~Lemma~\ref{interpolation} that $\wvec\in L^4(0, T_0;\Vvec)$ and
\[
	\|\wvec\|_{L^4(0, T_0;\Vvec)} \leq C\|\overline\yvec\|_{L^4(0, T_0;\widetilde\Vvec)}.
\]

	It is also clear that we can get estimates like those in the proof of Proposition~\ref{L-S-E-U-L-alpha-3D} for~$\widetilde\yvec$.
	In other words, for any $\yvec_0 \in \Vvec$, we can find a sufficiently small $T_0 > 0$ such that
\[
		\widetilde\yvec \in L^2(0, T_0;D(\widetilde\Avec)) \cap C^0([0, T_0];\widetilde\Vvec), \quad
		\widetilde\yvec_t\in L^2(0,T_0;\widetilde\Hvec)
\]
	and
\[
		\|\widetilde\yvec\|_{L^2(0, T_0;D(\widetilde\Avec))} + \|\widetilde\yvec\|_{C^0([0, T_0];\widetilde\Vvec)}
		+ \|\widetilde\yvec_t\|_{L^2(0,T_0;\widetilde\Hvec)}
		\leq C\left(T_0,\|\yvec_0\|_\Vvec,\|\overline\yvec\|_{L^4(0, T_0;\widetilde\Vvec)}\right),
\]
	where $C$ is nondecreasing with respect to all arguments and goes to zero as~$\|\yvec_0\|_\Vvec \to 0$.

	Now, let us introduce the mapping $\Phi_\alpha: L^4(0,T_0;\widetilde\Vvec) \mapsto L^4(0,T_0;\widetilde\Vvec)$, with $\Phi_\alpha(\overline \yvec)=\widetilde\yvec$ for all $\overline \yvec \in L^4(0,T;\widetilde\Vvec)$.
	This is a continuous and compact mapping.
	Indeed, from well known interpolation results, we have that the embedding
\[
	L^2(0, T_0;D(\widetilde\Avec)) \cap L^\infty(0, T_0;\widetilde\Vvec) \hookrightarrow L^4(0, T_0;D(\widetilde\Avec^{3/4}))
\]
	is continuous and this shows that, if $\widetilde\yvec$ is bounded in $L^2(0, T_0;D(\widetilde\Avec)) \cap C^0([0, T_0];\widetilde\Vvec)$ and~$\widetilde\yvec_t$ is bounded in $L^2(0,T_0;\widetilde\Hvec)$, then $\widetilde\yvec$ belongs to a compact set of~$L^4(0,T_0;\widetilde\Vvec)$.

	Then, as in the proofs of Theorems~\ref{NC-LERAY} and~\ref{NC-LERAY-BOUNDARY}, we immediately deduce that, whenever $\|\yvec_0\|_\Vvec \leq \delta$
	(for some $\delta$ independent of $\alpha$), $\Phi_\alpha$ possesses at least one fixed point.
	This shows that the nonlinear system \eqref{final-bound} is solvable for $\tilde \vvec\equiv0$ and $\|\yvec_0\|_\Vvec \leq \delta$.

	Now, the argument in the proof of Lemma \ref{regularity} can be applied in this framework and, as a consequence, we easily deduce Lemma~\ref{regularity_BC}.
\end{proof}

%%%%%%%%%%%%%%%%%%%%%%%%%%%%%%%%%%%%%%%%
%%%%%%%%%%%%%%%%%%%%%%%%%%%%%%%%%%%%%%%%
%%%%%%%% SECTION 5
%%%%%%%%%%%%%%%%%%%%%%%%%%%%%%%%%%%%%%%%
%%%%%%%%%%%%%%%%%%%%%%%%%%%%%%%%%%%%%%%%

\section{Additional comments and questions}\label{Sec5}

%%%%%%%%%%%%%%%%%%%%%%%%%%%%%%%%%%%%%%%%
%%%%%%%%%%%%%%%%%%%%%%%%%%%%%%%%%%%%%%%%
%%%%%%%% SUBSECTION
%%%%%%%%%%%%%%%%%%%%%%%%%%%%%%%%%%%%%%%%
%%%%%%%%%%%%%%%%%%%%%%%%%%%%%%%%%%%%%%%%

\subsection{Controllability problems for semi-Galerkin approximations}

	Let $\{\,\wvec^1, \wvec^2, \dots \,\}$ be a basis of the Hilbert space $\Vvec$.
	For instance, we can consider the orthogonal base formed by the eigenvectors of the Stokes operator $\Avec$.
	Together with \eqref{CLalpha}, we can consider the following semi-Galerkin approximated problems:
\begin{equation}\label{CLalpha-m}
	\left\{
		\begin{array}{lll}
			\yvec_t  -  \Delta \yvec +(\zvec^m \cdot \nabla) \yvec+ \nabla p = \vvec1_\om  											& \hbox{in} \ Q,    	 \\
			(\zvec^m(t),\wvec) \!+\! \alpha^2(\nabla\zvec^m(t),\wvec) = (\yvec(t),\wvec) \ \forall \wvec \in \Vvec_m, \ \zvec^m(t) \in \Vvec_m, \   & t \in (0,T), 		 \\
			\nabla \cdot \yvec=0,                                                  															&\hbox{in} \ Q,    	 \\
			\yvec = \ovec                                                                         	            												 &\hbox{on} \ \Sigma, 	 \\
			\yvec(0) = \yvec_0                                                                                       											 &\hbox{in} \ \Om,
		\end{array}
	\right.
\end{equation}
	where $\Vvec_m$ denotes the space spanned by $\wvec^1, \dots, \wvec^m$.

	Arguing as in the proof of~Theorem~\ref{NC-LERAY}, it is possible to prove a local null controllability result for~\eqref{CLalpha-m}.
	More precisely, for each $m\geq 1$, there exists $\eps_m > 0$ such that, if $\|\yvec_0\| \leq \eps_m$, we can find controls $\vvec^m$
	and associated states $(\yvec^m,p^m,\zvec^m)$ satisfying \eqref {null_condition}.
	Notice that, in view of the equivalence of norms in~$\Vvec_m$, the fixed point argument can be applied in this case without any extra
	regularity assumption on~$\yvec_0$; in other words, Lemma~\ref{regularity} is not needed here.
	
	On the other hand, it can also be checked that the maximal $\eps_m$ are bounded from below by some positive quantity independent of~$m$
	and~$\alpha$ and the controls $\vvec^m$ can be found uniformly bounded in~$L^\infty(0,T;\Lvec^2(\omega))$.
	As a consequence, at least for a subsequence, the controls converge weakly-$*$ in that space to a null control for~\eqref{CLalpha}.
	
	However, it is unknown whether the problems \eqref{CLalpha-m} are~{\it globally}~null-controllable;
	see below for other considerations concerning global controllability.

%%%%%%%%%%%%%%%%%%%%%%%%%%%%%%%%%%%%%%%%
%%%%%%%%%%%%%%%%%%%%%%%%%%%%%%%%%%%%%%%%
%%%%%%%% SUBSECTION
%%%%%%%%%%%%%%%%%%%%%%%%%%%%%%%%%%%%%%%%
%%%%%%%%%%%%%%%%%%%%%%%%%%%%%%%%%%%%%%%%

\subsection{Another strategy: applying an inverse function theorem}

	There is another way to prove the local null controllability of~\eqref{CLalpha} that relies on {\it Liusternik's Inverse Function Theorem,} see for instance~\cite{Alekseev}.
	This strategy has been introduced in~\cite{FURS-IMANU} and has been applied successfully to the controllability of many semilinear and nonlinear PDE's.
	In the framework of~\eqref{CLalpha}, the argument is as follows:
	
\begin{enumerate}

\item
	Introduce an appropriate Hilbert space $\Yvec$ of {\it state-control pairs} $(\yvec_\alpha,p_\alpha,\zvec_\alpha,\pi_\alpha,\vvec_\alpha)$ satisfying~\eqref{null_condition}.

\item
	Introduce a second Hilbert space $\Zvec$ of right hand sides and initial data and a well-defined mapping $\Fvec : \Yvec \mapsto \Zvec$ such that the null
	controllability of~\eqref{CLalpha} with state-controls in~$\Yvec$ is equivalent to the solution of the nonlinear equation
\begin{equation}\label{eqn-1}
	\Fvec(\yvec_\alpha,p_\alpha,\zvec_\alpha,\pi_\alpha,\vvec_\alpha) = ({\bf 0},\yvec_0), \quad (\yvec_\alpha,p_\alpha,\zvec_\alpha,\pi_\alpha,\vvec_\alpha)  \in \Yvec.
\end{equation}
	
\item
	Prove that $\Fvec$ is $C^1$ in a neighborhood of $({\bf 0},0,{\bf 0},0,{\bf 0})$ and $\Fvec'({\bf 0},0,{\bf 0},0,{\bf 0})$ is onto.
\end{enumerate}

	Arguing as in~\cite{FC-G-P}, all this can be accomplished satisfactorily.
	As a result, \eqref{eqn-1} can be solved for small initial data $\yvec_0$ and the local null controllability of~\eqref{CLalpha} holds.

%%%%%%%%%%%%%%%%%%%%%%%%%%%%%%%%%%%%%%%%
%%%%%%%%%%%%%%%%%%%%%%%%%%%%%%%%%%%%%%%%
%%%%%%%% SUBSECTION
%%%%%%%%%%%%%%%%%%%%%%%%%%%%%%%%%%%%%%%%
%%%%%%%%%%%%%%%%%%%%%%%%%%%%%%%%%%%%%%%%

\subsection{On global controllability properties}

	It is unknown whether a general global null controllability result holds for \eqref{CLalpha}.
	This is not surprising, since the same question is also open for the Navier-Stokes system.

	What  can be proved (as well as for the Navier-Stokes system) is the null controllability for large time:
	for any given $\yvec_0 \in \Hvec$, there exists $T_* = T_*(\|\yvec_0\|)$
	such that \eqref{CLalpha} can be driven exactly to zero with controls $\vvec_\alpha$ uniformly bounded in~$L^\infty(0,T_*;\Lvec^2(\omega))$.
	
	Indeed, let $\eps$ be the constant furnished by Theorem~\ref{NC-LERAY} corresponding to the time $T = 1$ (for instance).
	Let us first take $\vvec_\alpha \equiv \mathbf{0}$.
	Then, since the solution to~\eqref{Lalpha} with $\fvec = \mathbf{0}$ satisfies $\|\yvec_\alpha(t)\| \searrow 0$, there exists $T_0$
	(depending on~$\|\yvec_0\|$ but not on $\alpha$) such that $\|\yvec_\alpha(T_0)\| \leq \eps$.
	Therefore, there exist controls $\vvec'_\alpha \in L^\infty(T_0,T_0+1;\Lvec^2(\omega))$ such that the solution to~\eqref{CLalpha}
	that starts from $\yvec_\alpha(T_0)$ at time $T_0$ satisfies $\yvec_\alpha(T_0+1) = \mathbf{0}$.
	Hence, the assertion is fulfilled with $T_* = T_0 + 1$ and
\[
	\vvec_\alpha =
	\left\{
        		\begin{array}{ll}
			\dis
                		\ovec 		& \text{for \quad $0 \leq t < T_0$,}       \\
			\dis
                        \vvec'_\alpha    & \text{for \quad $T_0 \leq t \leq T_*$.}
		\end{array}
	\right.
\]
	
	A similar argument leads to the null controllability of~\eqref{CLalpha} for large~$\alpha$.
	In other words, it is also true that, for any given $\yvec_0 \in \Hvec$ and~$T > 0$, there exists $\alpha_* = \alpha_*(\|\yvec_0\|,T)$ such that,
	if $\alpha \geq \alpha_*$, then \eqref{CLalpha} can be driven exactly to zero at time~$T$.

\subsection{The Burgers-$\alpha$ system}

	There exist similar results for a regularized version of the Burgers equation, more precisely the Burgers-$\alpha$ system
\begin{equation}\label{B-alpha}
	\left\{
		\begin{array}{lll}
			y_t - y_{xx} + z y_x = v1_{(a,b)}                   		 	   & \hbox{in}& (0,L)\times(0,T), \\
			z - \alpha^2 z_{xx} = y                                			   & \hbox{in}& (0,L)\times(0,T), \\
			y(0,t) = y(L,t) = z(0,t) = z(L,t) = 0  		   & \hbox{on}& (0,T), \\
			y(x,0) = y_0(x)                                      				   & \hbox{in}& (0,L).
		\end{array}
	\right.
\end{equation}

	These have been proved in~\cite{FA-EFC-DAS}.
	
	This system can be viewed as a toy or preliminary model of~\eqref{CLalpha}.
	There are, however, several important differences between~\eqref{CLalpha} and~\eqref{B-alpha}:
	
\begin{itemize}

\item The solution to~\eqref{B-alpha} satisfies a maximum principle that provides a useful $L^\infty$-estimate.

\item There is no apparent energy decay for the uncontrolled solutions.
	As a consequence, the large time null controllability of~\eqref{B-alpha} is unknown.

\item It is known that, in the limit $\alpha = 0$, i.e.~for the Burgers equation, global null controllability does not hold;
	consequently, in general, the null controllability of~\eqref{B-alpha} with controls bounded independently of $\alpha$ is impossible.

\end{itemize}

	We refer to~\cite{FA-EFC-DAS} for further details.

%%%%%%%%%%%%%%%%%%%%%%%%%%%%%%%%%%%%%%%%
%%%%%%%%%%%%%%%%%%%%%%%%%%%%%%%%%%%%%%%%
%%%%%%%% SUBSECTION
%%%%%%%%%%%%%%%%%%%%%%%%%%%%%%%%%%%%%%%%
%%%%%%%%%%%%%%%%%%%%%%%%%%%%%%%%%%%%%%%%

\subsection{Local exact controllability to the trajectories}

	It makes sense to consider not only null controllability but also {\it exact to the trajectories} controllability problems for~\eqref{CLalpha}.
	More precisely, let $\hat\yvec_0 \in \Hvec$ be given and let $(\hat\yvec,\hat p,\hat\zvec,\hat\pi)$ a sufficiently regular solution to~\eqref{Lalpha} for
	$\fvec \equiv \mathbf{0}$ and $\yvec_0 = \hat\yvec_0$.
	Then the question is whether, for any given $\yvec_0 \in \Hvec$, there exist controls $\vvec$ such that the associated states, i.e.~the associated solutions to~\eqref{CLalpha}, satisfy
\[
	\yvec(T) = \hat\yvec(T)~\hbox{in}~\Om.
\]
	
	The change of variables
	\[
\yvec = \hat\yvec + \uvec, \ \zvec = \hat\zvec + \wvec,
	\]
	allows to rewrite this problem as the null controllability of a system similar, but not identical, to~\eqref{CLalpha}.
	It is thus reasonable to expect that a local result holds.

%%%%%%%%%%%%%%%%%%%%%%%%%%%%%%%%%%%%%%%%
%%%%%%%%%%%%%%%%%%%%%%%%%%%%%%%%%%%%%%%%
%%%%%%%% SUBSECTION
%%%%%%%%%%%%%%%%%%%%%%%%%%%%%%%%%%%%%%%%
%%%%%%%%%%%%%%%%%%%%%%%%%%%%%%%%%%%%%%%%

\subsection{Controlling with few scalar controls}

	The local null controllability with $N-1$ or even less scalar controls is also an interesting question.
	
	In view of the achievements in~\cite{C-G} and~\cite{Coron-Lissy} for the Navier-Stokes equations, it is reasonable to expect that results similar
	to~Theorems~\ref{NC-LERAY} and~\ref{CONVERGENCE} hold with controls $\vvec$ such that $v_i \equiv 0$ for some $i$;
	under some geometrical restrictions, it is also expectable that local exact controllability to the trajectories holds with controls of the same kind,
	see~\cite{FC-G-P-2}.
	
%%%%%%%%%%%%%%%%%%%%%%%%%%%%%%%%%%%%%%%%
%%%%%%%%%%%%%%%%%%%%%%%%%%%%%%%%%%%%%%%%
%%%%%%%% SUBSECTION
%%%%%%%%%%%%%%%%%%%%%%%%%%%%%%%%%%%%%%%%
%%%%%%%%%%%%%%%%%%%%%%%%%%%%%%%%%%%%%%%%

\subsection{Other related  controllability problems}

	There are many other interesting questions concerning the controllability of~\eqref{CLalpha} and related systems.
	
	For instance, we can consider questions like those above for the Leray-$\alpha$ equations completed with other boundary conditions:
	Navier, Fourier or periodic conditions for $\yvec$ and $\zvec$, conditions of different kinds on different parts of the boundary, etc.
	We can also consider Boussinesq-$\alpha$ systems, i.e. systems of the form
\[
	\left\{
		\begin{array}{lll}
		\yvec_t  -  \Delta \yvec + (\zvec \cdot \nabla) \yvec+ \nabla p = \theta \kvec + \vvec1_\om  	& \hbox{in} &       Q,    \\
		\theta_t  -  \Delta \theta + \zvec \cdot \nabla \theta = \wvec1_\om  	& \hbox{in} &       Q,     \\
		\zvec-\alpha^2\Delta \zvec +\nabla \pi=\yvec                                               	& \hbox{in} &       Q,     \\
		\nabla \cdot \yvec=0,~\nabla \cdot \zvec= 0                                                 	& \hbox{in} &       Q,     \\
		\yvec = \zvec=\ovec, \ \theta = 0                                                                          & \hbox{on}&  \Sigma, \\
		\yvec(0) = \yvec_0, \ \theta(0) = \theta_0                                                             & \hbox{in} &  \Om.
		\end{array}
	\right.
\]

	Some of these results will be analyzed in a forthcoming paper.\\ \\
\textbf{{ Acknowledgements}}\,\,\,
The authors thank J. L. Boldrini for the constructive conversations on the mathematical model.

%%%%%%%%%%%%%%%%%%%%%%%%%%%%%%%%%%%%%%%%
%%%%%%%%%%%%%%%%%%%%%%%%%%%%%%%%%%%%%%%%
%%%%%%%% REFERENCES
%%%%%%%%%%%%%%%%%%%%%%%%%%%%%%%%%%%%%%%%
%%%%%%%%%%%%%%%%%%%%%%%%%%%%%%%%%%%%%%%%

%\bibliographystyle{siam}
%\bibliography{biblio}

\end{document}